\documentclass[11pt]{amsart}
\usepackage{amsmath,amsfonts,amsbsy,amsgen,amscd,mathrsfs,amssymb,amsthm}
\usepackage{enumerate,mathtools}
\usepackage{tabularx}
\usepackage{array}
\newcolumntype{P}[1]{>{\centering\arraybackslash}p{#1}}
\usepackage{makecell}

\usepackage{stmaryrd}
\usepackage{fullpage}
\usepackage{url}
\usepackage{bm}
\usepackage{mathrsfs}

\usepackage[colorlinks=true,allcolors=blue]{hyperref}

\usepackage[
backend=biber,
style=alphabetic,
maxalphanames=5,
maxbibnames=99
]{biblatex}
\makeatletter
\renewcommand*{\eqref}[1]{%
  \hyperref[{#1}]{\textup{\tagform@{\ref*{#1}}}}%
}
\makeatother

\usepackage{tikz}
\usetikzlibrary{shadings}

\numberwithin{equation}{section}

\newtheorem{theorem}{Theorem}[section]
\newtheorem{proposition}[theorem]{Proposition}
\newtheorem{question}[theorem]{Question}
\newtheorem{conjecture}[theorem]{Conjecture}
\newtheorem{corollary}[theorem]{Corollary}
\newtheorem{lemma}[theorem]{Lemma}

\theoremstyle{definition}
\newtheorem{definition}[theorem]{Definition}

\newtheorem{remark}[theorem]{Remark}

\newtheorem*{claim*}{Claim}

\newcommand{\E}{\mathbb{E}}
\newcommand{\F}{\mathbb{F}}

\newcommand{\Q}{\mathbb{Q}}
\newcommand{\R}{\mathbb{R}}

\newcommand{\Z}{\mathbb{Z}}

\DeclareMathOperator{\Aut}{Aut}

\DeclareMathOperator{\Cov}{Cov}

\DeclareMathOperator{\Li}{Li}

\DeclareMathOperator{\Var}{Var}

\usepackage[shortlabels]{enumitem}

\numberwithin{figure}{section}

\title{Refinements on vertical Sato-Tate}

\author{Zhao Yu Ma}
\address{Department of Mathematics, Massachusetts Institute of Technology, 
Cambridge, MA, USA}
\email{zhaoyuma@mit.edu}

\begin{document}
\maketitle

\begin{abstract}
Vertical Sato-Tate states that the Frobenius trace of a randomly chosen elliptic curve over $\F_p$ tends to a semicircular distribution as $p\rightarrow \infty$. We go beyond this statement by considering the number of elliptic curves $N_{t,p}'$ with a given trace $t$ over $\F_p$ and characterizing the 2-dimensional distribution of $(t,N_{t,p}')$. In particular, this gives the distribution of the size of isogeny classes of elliptic curves over $\F_p$. Furthermore, we show a notion of stronger convergence for vertical Sato-Tate which states that the number of elliptic curves with Frobenius trace in an interval of length $p^{\epsilon}$ converges to the expected amount. The key step in the proof is to truncate Gekeler's infinite product formula, which relies crucially on an effective Chebotarev's density theorem that was recently developed by Pierce, Turnage-Butterbaugh and Wood. 
\end{abstract}

%\tableofcontents

\section{Introduction} 
Given an elliptic curve $E/\Q$, let $E_p/\F_p$ be the reduction of $E$ modulo a prime $p$. Assuming good reduction at $p$, we define the Frobenius trace to be $t_p(E)\coloneqq p+1 - \# E_p(\F_p)$ and the normalized Frobenius trace to be $x_p(E) = t_p(E)/\sqrt{p}$. By the Hasse bound, we have $x_p(E)\in [-2,2]$. For a fixed $E/\Q$, the Sato-Tate conjecture tells us how the normalized Frobenius traces $\{x_p(E)\}_{p}$ is distributed across primes $p$. Here we fixed an elliptic curve $E$ and varied the prime $p$, which is usually known as the \textit{horizontal} aspect of Sato-Tate. The Sato-Tate conjecture was proven by Barnet-Lamb, Geraghty, Harris and Taylor in \cite{sato-tate}.
%\cite{delabretèche2018satotate}

On the other hand, one could do the opposite -- fix $p$ and consider all elliptic curves $E/\F_p$. This is known as the \textit{vertical} aspect of Sato-Tate, which we will be focusing on in this paper. For each $p$, let $\mu_p$ be the distribution of $x_p(E)$ where $E$ is picked uniformly at random among all (isomorphism classes of) elliptic curves over $\F_p$. Birch computed the moments of $\mu_p$ and showed that the distributions $\mu_p$ converge as $p\rightarrow \infty$:
\begin{theorem}[Vertical Sato-Tate, \cite{Birch1968HowTN}]\label{thm: vertical sato-tate}
Define $\mu$ to be the semicircular distribution supported on $[-2,2]$, i.e. $\mu$ has probabilistic mass function $\frac{1}{2\pi}\sqrt{4-x^2}$. Then, $\mu_p$ converges in distribution to $\mu$ as $p\rightarrow \infty$. Equivalently, for every subinterval $[a,b]$ of $[-2,2]$, we have 
\begin{equation}\label{eqn: integral form vertical sato-tate}
\lim_{p\rightarrow \infty} \frac{\#\{E/\F_p \mid t_p(E)\in [a\sqrt p, b\sqrt p]\}}{\#\{E/\F_p\}} = \frac{1}{2\pi} \int_a^b \sqrt{4-x^2}dx.
\end{equation}
\end{theorem}
Visually, for each fixed $p$, we can plot a histogram of $x_p(E)$ over all $E/\F_p$ with a constant number of bins, as in Figure \ref{fig: constant bins}. Theorem \ref{thm: vertical sato-tate} is equivalent to the fact that the histograms converges uniformly as $p\rightarrow \infty$, as we can see in the figure. A natural question to ask is the following.
\begin{question}\label{question}
Will there still be such a convergence if the number of bins in the histogram increases with $p$?
\end{question}
\begin{figure}
    \centering
    \includegraphics[scale = 0.5]{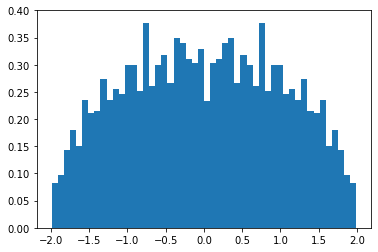}
    \includegraphics[scale = 0.5]{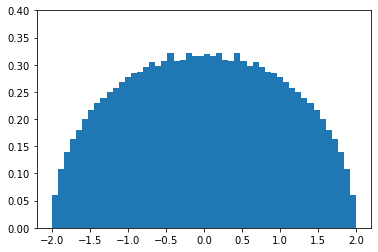}
    \caption{Histogram of $x_p(E)$ over all $E/\F_p$ for $p=10007$ and $p=1000003$ respectively, each with $50$ bins. We observe better convergence to a semicircular distribution as $p\rightarrow \infty$.}
    \label{fig: constant bins}
\end{figure}
To address this question, we first consider the extreme case where there are $4\sqrt{p}$ bins, one for each unique Frobenius trace. Let $N_{t,p}'$ be the number of elliptic curves $E/\F_p$ with Frobenius trace $t_p(E)=t$, so the $t$-th bin has a count of $N_{t,p}'$ elliptic curves. We view this histogram via the scatter plot of the normalized points $(t/\sqrt p, N_{t,p}'/2\sqrt p)$ for $t\in [-2\sqrt{p},2\sqrt{p}]$ as seen in Figure \ref{fig: scatter plot}. Visually, there is no convergence to a semicircular distribution. Indeed, \cite{refinements} showed that $N_{t,p}'/2\sqrt{p}$ can be arbitrarily large, implying that there is no convergence.

\begin{figure}
    \centering
    \includegraphics[scale = 0.5]{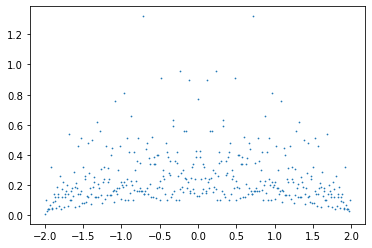}
    \includegraphics[scale = 0.5]{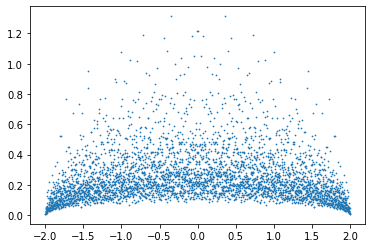}
    \caption{Scatter plot of the normalized points $(t/\sqrt p, N_{t,p}'/2\sqrt p)$ for $t\in [-2\sqrt{p},2\sqrt{p}]$ for $p=10007$ and $p=1000003$ respectively. This reflects the histogram with $4\sqrt p$ bins. We observe no convergence to a semicircular distribution as $p\rightarrow \infty$.}
    \label{fig: scatter plot}
\end{figure}

In this paper, we go further than the results in \cite{refinements} by estimating the two-dimensional distribution of the scatter plot $(t/\sqrt p, N_{t,p}'/2\sqrt p)$ for each prime $p$. There is a simple heuristic for this which we explain here briefly (see Section \ref{sec: heuristics 1} for more details), stemming from Gekeler's infinite product formula \cite{Gekeler2003FrobeniusDO},
$$N_{t,p}' \approx \frac{\sqrt{4p-t^2}}{\pi}\prod_{l \text{ prime}}v_l(t^2-4p),$$
where the factors $v_l$ are defined in Equation \eqref{eqn: v_l original}. Let $Y_{l,p}$ be the random variable of $v_l(t^2-4p)$ where $t$ is a ``random integer" (see Definition \ref{defn: Y_l,p}). We heuristically expect that $v_l(t^2-4p)$ are independent for each $l$, so let us define the error term $Z_{p}$ as the infinite independent product of the $Y_{l,p}$ (see Proposition \ref{prop: def Z_p}). Then, we expect
$$N_{t,p}' \sim \frac{\sqrt{4p-t^2}}{\pi}Z_p.$$
It is worth noting that we can compute the explicit distribution of $Y_{l,p}$ as in Lemma \ref{lem: Y_l,p explicit distributions}, and thus $Z_p$ is computable as well, albeit as an infinite product. Thus, it follows that if we pick a random trace $t\in [-2\sqrt p,2\sqrt p]$ we should heuristically expect that
$$(t/\sqrt{p}, N_{t,p}'/2\sqrt{p}) \sim (X,\frac{1}{2\pi}\sqrt{4-X^2}Z_p)$$
where $X$ is the uniform random variable on $[-2,2]$, and furthermore $X$ and $Z_p$ are assumed to be independent. Our first main result shows that this heuristic estimate is correct in the $p\rightarrow \infty$ limit:
\begin{theorem}(2D vertical Sato-Tate)\label{thm: 2D vertical sato-tate}
Define $\rho_p^{heu}$ to be the joint distribution of $(X,\frac{1}{2\pi}\sqrt{4-X^2}Z_p)$, where $X$ is the uniform random variable on $[-2,2]$, and additionally $X$ and $Z_p$ are independent random variables. Define $\rho_p$ to be the joint distribution of $(t/\sqrt p, N_{t,p}'/2\sqrt p)$ where $t\in [-2\sqrt{p},2\sqrt{p}]$ is chosen randomly. Then
$$\lim_{p\rightarrow \infty} \pi(\rho_p,\rho_p^{heu}) = 0$$
where $\pi$ is the Lévy–Prokhorov metric defined in Definition \ref{def: levy-prokhorov}.
\end{theorem}
The statement of this theorem may seem rather technical, and so we will attempt to give a more intuitive explanation via an analogy. Suppose we have two sequence of real numbers $a_i = (\pm 1)^n + \frac 1 n$, and $b_i = (\pm 1)^n$. We see that $\lim_{i\rightarrow \infty} d(a_i - b_i) = 0$, so $a_i$ and $b_i$ ``converges" as $i\rightarrow \infty$. Informally, we can say that ``$a_i\rightarrow b_i$" as $i\rightarrow \infty$. In our context, convergence in the Lévy–Prokhorov metric on $\R^2$ is equivalent to convergence in distribution for random variables (i.e. weak convergence of distributions). In the same way as the case for real numbers, the conclusion of Theorem \ref{thm: 2D vertical sato-tate} that $\lim_{p\rightarrow \infty} \pi(\rho_p,\rho_p^{heu})=0$ means that the actual distribution $\rho_p$ and the heuristic distribution $\rho_p^{heu}$ ``converges" in distribution. Informally, we can say that ``$\rho_p\rightarrow \rho_p^{heu}$" in distribution as $p\rightarrow \infty$ just as in the analogous case in real numbers. 

What we would have really liked to say was that the distributions $\rho_p$ converge in distribution to some fixed 2D distribution as $p\rightarrow \infty$. This seems visually plausible given the heat maps of the normalized points $(t/\sqrt p, N_{t,p}'/2\sqrt p)$ in Figure \ref{fig: heat maps}. Unfortunately, that is not the case in reality, and this resulted in a more technical looking statement for Theorem \ref{thm: 2D vertical sato-tate}. To explain this, our result in Theorem \ref{thm: 2D vertical sato-tate} tells us that $\rho_p$ converges if and only if the heuristic distributions $\rho_p^{heu}$ converge. By definition of $\rho_p^{heu}$, we see that this happens if and only if $Z_p$ converges. However, the distributions of $Z_p$ differ very slightly from prime to prime, and one can show that $Z_p$ does not converge in distribution as $p\rightarrow \infty$. Thus, surprisingly, the distributions $\rho_p$ do not actually converge to a fixed 2D distribution. In fact, the distributions are so similar that it is difficult to ascertain the lack of convergence by looking at figures with the naked eye.

Instead, we give a sufficient condition on sequences of primes $p_i$ where $\rho_{p_i}$ converges. By the same argument above, $\rho_{p_i}$ converges if and only if $Z_{p_i}$ converges. The distribution of $Z_p$ depends only on the values of the Legendre symbol $(\frac p l)$ over all odd primes $l$ and the value of $p$ modulo $8$. Furthermore, the bigger the prime $l$, the less $Z_p$ will be affected by the value of $(\frac p l)$. From these observations, we can derive the following corollary.
\begin{figure}
    \centering
    \includegraphics[scale = 0.5]{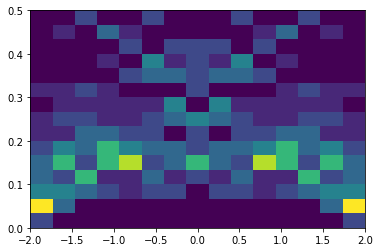}
    \includegraphics[scale = 0.5]{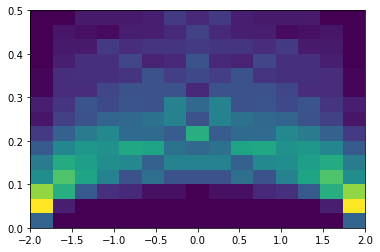}
    \caption{Heat map of the normalized points $(t/\sqrt p, N_{t,p}'/2\sqrt p)$ for $t\in [-2\sqrt{p},2\sqrt{p}]$ for $p=10007$ and $p=1000003$ respectively. We observe that the points $(t/\sqrt p, N_{t,p}'/2\sqrt p)$ seem to converge to a 2D distribution. However, this is only true for certain sequences of primes by Corollary \ref{cor: prime sequence converge}.}
    \label{fig: heat maps}
\end{figure}
\begin{corollary}\label{cor: prime sequence converge}
Let $p_1,p_2,\ldots$ be an increasing sequence of primes. Then $\rho_{p_i}$ converges in distribution if for each odd prime $l$ the Legendre symbols $(\frac {p_i} l)$ are eventually constant, and additionally one of the following is true: (1) $p_i$ is eventually $3 \pmod 4$, (2) $p_i$ is eventually $1 \pmod 8$, or (3) $p_i$ is eventually $5 \pmod 8$.
\end{corollary}
We conjecture that the converse of Corollary \ref{cor: prime sequence converge} is true. We could not prove this, but we reduced this to proving that a family of infinite products of distributions are distinct, see Section \ref{sec: convergent subsequences} for more details. Furthermore, by projecting onto the second component in Theorem \ref{thm: 2D vertical sato-tate}, we obtain results on the distribution of the size of isogeny classes of elliptic curves over a fixed prime field in the $p\rightarrow \infty$ limit, see Corollary \ref{cor: isogeny classes}.

Theorem \ref{thm: 2D vertical sato-tate} does not follow easily from the heuristic argument given above. There are two main difficulties in proving Theorem \ref{thm: 2D vertical sato-tate} (1) there are infinitely many factors $v_l(t^2-4p)$ in the formula for $N_{t,p}'$ and (2) showing independence between the factors $v_l(t^2-4p)$. The key step in the proof is to show that Gekeler's infinite product formula can be truncated and still give a good approximation of $N_{t,p}'$. This uses an effective Chebotarev's density theorem developed recently by \cite{PTW20}. From this, we obtain the truncation formulas Proposition \ref{prop: first step}, \ref{prop: second step} and \ref{prop: chebyshev result}. 

Now we discuss the second main result in this paper. We answered Question \ref{question} in the negative -- although vertical Sato-Tate tells us that the histograms with a constant number of bins each of size $\Theta(\sqrt p)$ converges uniformly to a semicircular distribution, Theorem \ref{thm: 2D vertical sato-tate} tells us that the histograms with $4\sqrt p$ bins each with a single trace does not converge to a semicircular distribution. Then, the natural followup question to ask would be what the threshold for convergence is. More precisely, we ask for what functions $f(p)=o(\sqrt p)$ does
\begin{equation}\label{eqn: threshold question}
\sup_{-2\sqrt p\le t\le 2\sqrt{p}-f(p)}\left| \frac{\#\{E/\F_p \mid t_p(E)\in \left[t,t+f(p)\right)\}}{\frac{f(p)}{\sqrt p}\#\{E/\F_p\}} - \frac{1}{2\pi}\sqrt{4-(t/\sqrt p)^2}\right|\rightarrow 0
\end{equation}
as $p\rightarrow \infty$? 

The above equation corresponds to uniform convergence of the histogram with $4\sqrt{p}/f(p)$ bins each of size $f(p)$ to a semicircular distribution. Our second main result shows that there is still convergence even when the bin size $f(p)$ grows much slower than $\sqrt p$. This implies a stronger convergence result in comparison to the original vertical Sato-Tate Theorem \ref{thm: vertical sato-tate}.
\begin{theorem}(Strong vertical Sato-Tate)\label{thm: stronger vertical sato-tate}
Equation \eqref{eqn: threshold question} is true if there exists some $\epsilon >0$ where $f(p)\ge p^\epsilon$ for all sufficiently large $p$. If we assume the Generalized Riemann Hypothesis (GRH), Equation \eqref{eqn: threshold question} is true if for every constant $C>0$, $f(p)\ge \exp(C(\log\log(p))^2)$ for all sufficiently large $p$.
\end{theorem}
\begin{figure}
    \centering
    \includegraphics[scale = 0.5]{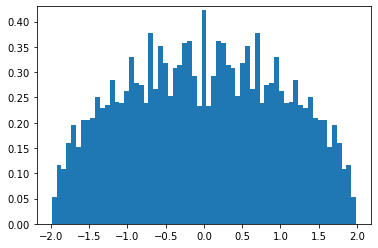}
    \includegraphics[scale = 0.5]{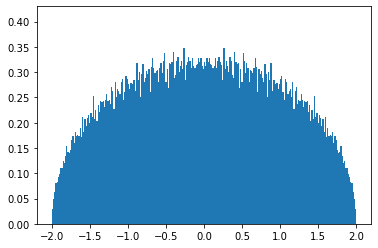}
    \caption{Histogram of $x_p(E)$ over all $E/\F_p$ for $p=10007$ and $p=1000003$ respectively. There are $4p^{0.3}$ bins, each counting the number of elliptic curves with Frobenius trace $t_p(E)$ within a $p^{0.2}$ range. As predicted by Theorem \ref{thm: stronger vertical sato-tate}, we observe convergence to a semicircular distribution.}
    \label{fig: changing bins}
\end{figure}
For the unconditional statement, $p^{\epsilon}$ comes from the statement of the effective Chebotarev's density theorem in \cite{PTW20} that there are at most $O(D^{\epsilon})$ exceptional fields with discriminant less than $D$. This is the only limiting factor -- a stronger effective Chebotarev's density theorem with fewer exceptional fields would improve this bound. When we assume GRH for the latter part of this theorem and throughout this paper, we actually only need the special case of GRH for the Dedekind zeta function of quadratic fields. 

For reference, we plot the histograms where each bin counts the number of elliptic curves within a range of $p^{0.2}$ Frobenius traces in Figure \ref{fig: changing bins}. It is worth noting that the method of proving vertical Sato-Tate in \cite{Birch1968HowTN} by computing moments does not extend to a proof of this stronger convergence. Our proof of Theorem \ref{thm: stronger vertical sato-tate} also stems from the same technique of truncating Gekeler's infinite product formula. A proof outline for both Theorem \ref{thm: 2D vertical sato-tate} and \ref{thm: stronger vertical sato-tate} is given in Section \ref{sec: proof overview}, and the full details of the proof are given in Sections \ref{sec: truncation formula} to \ref{sec: pf of stronger sato-tate}.

Lastly, we discuss possible extensions to higher genus curves as well as abelian varieties. The higher genus case was discussed in \cite{refinements}, and interestingly, the higher genus case exhibits a different behaviour from the genus $1$ case. Let $N_{t,p,g}'$ be the number of genus $g$ curves of Frobenius trace $t$ over $\F_p$. For $g\ge 3$, \cite{refinements} conjectured that each $N_{t,p,g}'$ converges to the corresponding Sato-Tate curve, and for $g=2$, there is convergence to two separate Sato-Tate curves depending on the parity of $t$. For abelian varieties, \cite{Achter_2023} proved a more general formula that computes the size of the isogeny class of principally polarized abelian varieties as an infinite product, which specializes to Gekeler's formula in the dimension $1$ case. It is likely that the methods in this paper can be generalized to principally polarized abelian varieties to prove similar distribution results, including results on the distribution of the size of isogeny classes of principally polarized abelian varieties.
\subsection{Acknowledgments} 
I would like to thank Andrew Sutherland for mentoring me for this project, and for providing valuable guidance and advice. I would also like to thank Jeff Achter, Aaron Landesman, Christophe
Ritzenthaler, Jesse Thorner, Jit Wu Yap and the anonymous referee for their helpful comments.
\section{Preliminaries} 
\subsection{Elliptic curve counts} It is more convenient to count elliptic curves $E/\F_p$ weighted according to $\frac{2}{\#\Aut_{\F_p}(E)}$, and we will do this from now on. Define 
$N_{t,p}$ to be the number of elliptic curves $E/\F_p$ with Frobenius trace $t_p(E)=t$ weighted according to $\frac{2}{\#\Aut_{\F_p}(E)}$. As noted in \cite{Gekeler2003FrobeniusDO}, for any prime $p$, there are at most 10 elliptic curves $E/\F_p$ with $\#\Aut_{\F_p}(E)\neq 2$. Thus, $N_{t,p}$ and $N_{t,p}'$ differ by at most a constant which becomes negligible for big $p$, so it suffices to consider the weighted counts $N_{t,p}$ instead of the actual counts $N_{t,p}'$. One can obtain the following expression for $N_{t,p}$.
\begin{proposition}[\cite{Gekeler2003FrobeniusDO}]\label{prop: N_t,p formula}
We have the formula
$$N_{t,p} = H(\Delta) = \frac{\sqrt{-\Delta}}{\pi}\prod_{l \text{ prime}}v_l(\Delta)$$
where $\Delta = t^2-4p$, $H(\Delta)$ is the Kronecker-Hurwitz class number and the weights $v_l(\Delta)$ are defined as
\begin{equation}\label{eqn: v_l original}
v_l(\Delta) = (1-l^{-2})^{-1}\left(1+l^{-1} + \begin{Bmatrix}
    0 \\
    -(l+1)l^{-\delta-2}\\
    -2l^{-\delta-1}
\end{Bmatrix}\right)\end{equation}
where $\delta$ is the largest non-negative integer such that $l^{2\delta}\mid \Delta$, and in the case $l=2$, we require additionally that $\Delta/2^{2\delta}\equiv 0,1 \mod{4}$. The curly brackets correspond to the cases where the Kronecker symbol $\left(\frac{\Delta/l^{2\delta}}{l}\right)$ is $+1,0,-1$. 
\end{proposition}
\subsection{Weights $v_l(\Delta)$}
From now on, let $\Delta = \Delta_p(t) \coloneqq t^2-4p$ be a function of $t$, so that $v_l(\Delta)$ is implicitly a function of $t$. The following facts follow directly from the definition.
\begin{lemma}\label{lem: v_l facts}
The following statements are true:
\begin{enumerate}[(\alph*)]
\item If $l^2 \nmid \Delta$, then $\delta =0$ so
$v_l(\Delta) = \left( 1 - \left( \frac{\Delta}{l}\right)/l\right)^{-1}$.
\item We have the bound 
$\left( 1 + \frac 1 l \right)^{-1} \le v_l(\Delta) \le \left( 1 - \frac 1 l \right)^{-1}$. Thus, we have $v_l(\Delta) = 1+O(l^{-1})$ and $\log(v_l(\Delta)) = O(l^{-1})$.
\end{enumerate}
\end{lemma}
By definition, the value of $v_l(\Delta)$ only depends on the residues $\Delta$ modulo $l^{i}$ for each $i$. In other words, the function $v_l$ extends to a continuous function on $\Z_l$. This allows us to make the following important definition.
\begin{definition}\label{defn: Y_l,p}
Let $Y_{l,p}$ be the random variable with value $v_l(\Delta) = v_l(t^2-4p)$, where $t$ is chosen uniformly at random in $\Z_l$ according to its Haar measure.
\end{definition}
It is not difficult to explicitly calculate the distributions of each $Y_{l,p}$ by definition, and we state this in Lemma \ref{lem: Y_l,p explicit distributions} in the Appendix. We have the following important observations which follow directly from Lemma \ref{lem: Y_l,p explicit distributions}. The last statement of Lemma \ref{lem: Y_l,p expectation variance} that $\E[Y_{l,p}]=1$ for $p\neq l$ has also been proven in pages 9-10 of \cite{galbraith_mckee_2000}.
\begin{lemma}\label{lem: Y_l,p depends on (p/l)}
For each fixed odd $l$, the distribution of $Y_{l,p}$ only depends on value of the Legendre symbol $(\frac p l)$. For $l=2$, the distribution of $Y_{2,p}$ only depends on the value of $p$ modulo $8$.
\end{lemma}
\begin{lemma}\label{lem: Y_l,p expectation variance}
We have $$\E[Y_{l,p}]=1+O(l^{-1}),\quad \Var[Y_{l,p}] = O(l^{-2}),\quad \E[\log(Y_{l,p})]=O(l^{-2}),\quad \Var[\log(Y_{l,p})] = O(l^{-2}).$$
Furthermore, if $p\neq l$, then 
$\E[Y_{l,p}]=1$.
\end{lemma}
\subsection{Lévy–Prokhorov metric}\label{sec: prob theory} Here we define and state some properties of the Lévy–Prokhorov metric. 
\begin{definition}[Lévy–Prokhorov metric]\label{def: levy-prokhorov}
Let $(M,d)$ be a metric space with its Borel sigma-algebra $\mathcal B(M)$, and let $\mathcal P(M)$ be the collection of all probability measures on the measurable space $(M,\mathcal B(M))$. For a subset $A\subseteq M$ and an $\epsilon>0$, let $A^{\epsilon}$ be the $\epsilon$-neighborhood of $A$. Define the Lévy–Prokhorov metric $\pi$ on $\mathcal P(M)$ as
$$\pi(\mu,\nu) \coloneqq \inf \{\epsilon > 0 \mid \mu(A)\le \nu(A^\epsilon)+\epsilon, \ \nu(A)\le \mu(A^\epsilon)+\epsilon \text{ for all } A\in \mathcal B(M)\}.$$
\end{definition}
We will only be concerned with the case $(M,d) = (\R^n,d)$ where $d$ is the Euclidean metric. In $(\R^n,d)$, convergence of measures in the Lévy–Prokhorov metric is equivalent to weak convergence of measures, or equivalently, convergence in distribution for random vectors. 

For a random vector $X$, denote $\mu_X$ as its probability distribution on $\R^n$. In the following cases we have bounds on the Lévy–Prokhorov distance between the probability distributions of two random vectors.
\begin{lemma}\label{lem: bound metric}
Suppose $X$ and $Y$ are joint random vectors on $\R^n$. If $\mathbb P(|X-Y|\ge \epsilon)\le \epsilon$, then $\pi(\mu_X,\mu_{Y})\le \epsilon$.
\end{lemma}
\begin{proof}
For any $A\in \mathcal B(\R^n)$, we want to show that $\mu_X(A) = \mathbb P(X\in A) \le \mathbb P(Y\in A^\epsilon) + \epsilon = \mu_Y(A^\epsilon)+\epsilon$. To see this, note that if $X\in A$, we either have $|X-Y|< \epsilon$ in which case $Y\in A^\epsilon$, or $|X-Y|\ge \epsilon$. Thus, $\mathbb P(X\in A) \le \mathbb P(Y\in A^\epsilon) + \mathbb P(|X-Y|\ge \epsilon)\le \mathbb P(Y\in A^\epsilon) + \epsilon$.
\end{proof}
\begin{corollary}\label{cor: variance metric}
If $X$ and $Y$ are joint random variables on $\R$ where $|\E[X-Y]|\le \delta$ and $\Var[X-Y]\le \epsilon$, then $\pi(\mu_X,\mu_Y)\le \delta+\epsilon^{1/3}$.
\end{corollary}
\begin{proof}
By Chebyshev's inequality we have $\mathbb P(|X-Y-\E[X-Y]|\ge \epsilon^{1/3})\le \epsilon^{1/3}$, so by the triangle inequality $\mathbb P(|X-Y|\ge \delta+\epsilon^{1/3})\le \epsilon^{1/3}$ and the statement follows by Lemma \ref{lem: bound metric}.
\end{proof}
We need to show that a sequence of random variables converges in distribution to some random variable. The following lemma says it suffices to show that it is Cauchy in the Lévy–Prokhorov metric.
\begin{lemma}\label{lem: metric complete}
Suppose $X_1,X_2,\ldots$ is a sequence of random variables where $\mu_{X_i}$ is Cauchy in the Lévy–Prokhorov metric $\pi$. Then $X_i$ converges in distribution to some random variable $X$.
\end{lemma}
\begin{proof}
If $(M,d)$ is separable, it is well-known that $(\mathcal P(M),\pi)$ is complete if and only if $(M,d)$ is complete. Since $(\R,d)$ is both complete and separable, $(\mathcal P(\R),\pi)$ is complete. Thus, $\mu_{X_i}$ converge to some probability measure $\mu$ on $\R$. There is a one-to-one bijection between probability measures $\mu$ on $\R$ and distribution functions $F(x)$. Every distribution function $F(x)$ comes from a random variable. Thus, $\mu = \mu_X$ for some random variable $X$. Then $X_i$ converges in distribution to $X$ as convergence in the metric $\pi$ is equivalent to convergence in distribution for random variables.
\end{proof}
%https://encyclopediaofmath.org/wiki/L%C3%A9vy-Prokhorov_metric
%https://encyclopediaofmath.org/wiki/Distribution_function
\subsection{Heuristics}\label{sec: heuristics 1}
Recall that $N_{t,p}$ is given by the infinite product in Proposition \ref{prop: N_t,p formula}. Fix a prime $p$ and let us pick a trace $t$ uniformly at random in $[a\sqrt p,b\sqrt p]$ where $[a,b]$ is a subinterval of $[-2,2]$. Then, as $p\rightarrow \infty$, the distribution of $v_l(\Delta)$ (as a random variable of $t$) should be almost the same as the random variable $Y_{l,p}$ because the image of $[a\sqrt p,b\sqrt p]$ under the map $\Z \rightarrow \Z_l$ is almost equidistributed. Heuristically, we also expect the image of $\Delta$ in $\Z_l$ to be independent for each $l$, so the distributions of $v_l(\Delta)$ should be independent for each $l$. Thus, we should expect for a random trace $t\in [a\sqrt p, b\sqrt p]$ that
$$N_{t,p} \sim \frac{\sqrt{4p-t^2}}{\pi}\prod_{l \text{ prime}}Y_{l,p}$$
where in this equation we take the infinite independent product of random variables $Y_{l,p}$.

It is important to clarify what we mean by the independent product of random variables, especially since each $Y_{l,p}$ lives on a different event space. Suppose we have random variables $X$ and $Y$ living on the event spaces $\mathcal F_X$ and $\mathcal F_Y$ respectively. Define the independent product as the random variable taking the value $X\times Y$ living on the product of the $\sigma$-algebras $\mathcal F_X \times \mathcal F_Y$. We can define the independent product, as well as independent sum, over a finite index set this way. But if the index set is infinite, we may have problems with issues of convergence. For example, if $Y_{l,p}$ takes the value of $(1-\frac{1}{l})^{-1}$ for each prime $l$, then $\prod_l (1-\frac{1}{l})^{-1}$ is not well defined. To overcome this, we define the independent product over an ordered index set to be the random variable with probability distribution equal to the limit of the distributions of the finite products, where we have to check convergence of distributions. The following proposition shows that this is well defined for the case of $Y_{l,p}$, which is the only case where an infinite independent product will appear. 
\begin{proposition}\label{prop: def Z_p}
The distribution of the finite independent products $\prod_{l<n \text{ prime}}Y_{l,p}$ converges in distribution to some random variable, which we define to be $Z_p = \prod_l Y_{l,p} \coloneqq \lim_n (\prod_{l<n \text{ prime}}Y_{l,p})$.
\end{proposition}
\begin{proof}
All products and summations in this proof are over primes. Since $\prod_{l} Y_{l,p} = \exp(\sum_{l} \log(Y_{l,p}))$, it suffices to show that $\sum_{l} \log(Y_{l,p})$ converges in distribution to some random variable. By Lemma \ref{lem: metric complete} it suffices to show that $\mu_{X_{m}}$ is Cauchy with respect to $\pi$, where $X_m = \sum_{l<m} \log(Y_{l,p})$. 

Consider some $N$. For all $m_1>m_2>N$ we have $\E[X_{m_1}-X_{m_2}] = \sum_{m_1\le l <m_2} \log(Y_{l,p}) = \sum_{m_1\le l <m_2} O(l^{-2}) = O(N^{-1})$ by Lemma \ref{lem: Y_l,p expectation variance}. Likewise, using independence, $\Var[X_{m_1}-X_{m_2}] = O(N^{-1})$ in the same way. By Corollary \ref{cor: variance metric}, $\pi(\mu_{X_{m_1}},\mu_{X_{m_2}})=O(N^{-1/3})$, so $\mu_{X_{m}}$ is indeed Cauchy.
\end{proof}
Thus, we have 
$$N_{t,p} \sim \frac{\sqrt{4p-t^2}}{\pi}Z_{p}.$$
As we expect $N_{t,p}$ on average to be $ \frac{\sqrt{4p-t^2}}{\pi}$, we call $Z_p$ the error term. This provides a heuristic argument for Theorem \ref{thm: 2D vertical sato-tate}.
% Here we collect some properties of $Z_p$.
% \begin{lemma}\label{lem: expectation Z_p}
% We have $\E[Z_p]=1+O(p^{-1})$ and $\Var[Z_p]=\Theta(1)$
% \end{lemma}
% \begin{proof}
% By independence, $\E[Z_p] = \prod_{l \text{ prime}} \E[Y_{l,p}] = \E[Y_{p,p}]$ by Lemma \ref{lem: Y_l,p expectation variance}. The first statement then follows from the explicit description of $Y_{p,p}$ in Lemma \ref{lem: Y_l,p explicit distributions}. 
% \end{proof}
% \begin{lemma}

% \end{lemma}
\section{Proof overview}\label{sec: proof overview}
The proof is split into four steps, and we explain each of the steps in the following subsections. Note that everything in this section is meant to be informal and intuitive, but not rigorous. The detailed, rigorous proof will be in the following sections, organized as follows. The first two steps on truncation and approximating the factors are carried out in Section \ref{sec: truncation formula}. The application to 2D vertical Sato-Tate (Theorem \ref{thm: 2D vertical sato-tate}, Corollary \ref{cor: prime sequence converge}) is in Section \ref{sec: pf of 2D sato-tate}, and the application to stronger vertical Sato-Tate (Theorem \ref{thm: stronger vertical sato-tate}) is in Section \ref{sec: pf of stronger sato-tate}.

Now we begin the informal sketch of the proof. The two main difficulties of the proof are that (1) there are infinitely many factors $v_l(t^2-4p)$ in the formula for $N_{t,p}$ and (2) we need to show independence between the factors $v_l(t^2-4p)$. 
\subsection{Truncation} We address the first difficulty mentioned above by truncating the infinite product to a finite product. Consider the tail end of the infinite product from Proposition \ref{prop: N_t,p formula} $\prod_{l\ge l_0} v_l(\Delta)$, with a suitable choice of cutoff $l_0=l_0(p)$.  Throughout the rest of the paper, such a product will always implicitly be over primes $l$. We want to show that 
$$\prod_{l\ge l_0}v_l(\Delta) = 1+o(1)$$
which implies that the truncated product is a good enough approximation to the infinite product
\begin{equation}\label{eqn: truncation estimate}
N_{t,p} = (1+o(1)) \frac{\sqrt{-\Delta}}{\pi}\prod_{l < l_0} v_l(\Delta).
\end{equation}
By Lemma \ref{lem: v_l facts}, write
\begin{equation*}
\begin{split}
    \prod_{l\ge l_0} v_l(\Delta) &= \prod_{l\ge l_0, \ l \nmid \Delta} v_l(\Delta) \prod_{l\ge l_0, \ l \mid \Delta} v_l(\Delta) \\
    &= \prod_{l\ge l_0, \ l \nmid \Delta} \left( 1 - \left( \frac{\Delta}{l}\right)/l\right)^{-1} \prod_{l\ge l_0, \ l \mid \Delta} v_l(\Delta).
\end{split}
\end{equation*}

We reason that this product converges using its factorization into the two products above. The second product only has finitely many terms, so we just need to show that the first product converges. Let $\Delta = f^2 \Delta_0$ where $\Delta_0$ is a fundamental discriminant. Then $(\frac{\Delta}{l}) = (\frac{\Delta_0}{l})(\frac{f^2}{l}) = (\frac{\Delta_0}{l})=\chi_0(l)$ where $\chi_0$ is the quadratic character associated to $\Delta_0$. In the second equality we used the fact that $l\nmid f$ because $l \nmid \Delta$. Hence, the first product is the Euler product of a Dirichlet L-function $L(1,\chi_0)$ with finitely many terms omitted, so it converges.

Consider the imaginary quadratic number field $K_{\Delta}=\Q(\sqrt{\Delta_0})$ with discriminant $\Delta_0$. Then, $(\frac{\Delta}{l})=1$ if the prime $l$ splits in $K_{\Delta}$, otherwise it is $-1$. By Chebotarev's density theorem, each case appears about half the time, so the terms in the first product should almost always cancel out. To give a concrete bound we will use two effective versions of Chebotarev's density theorem, one assuming GRH by \cite{LO75}, and an unconditional one by \cite{PTW20}. This will help us bound the difference in the number of $+1$s and $-1$s for each dyadic interval $[2^n,2^{n+1}]$. Thus we can show that $\prod_{l\ge l_0} v_l(\Delta)$ is almost always small when $l_0$ is big enough relative to the size of the discriminant $|\Delta_0|$. Since $|\Delta_0|\le |\Delta|\le p+1+2\sqrt p$, it suffices to have a lower bound for $l_0$ in terms of $p$.

Going through this argument gives us the following proposition.
\begin{proposition}\label{prop: first step}
Assuming GRH, Equation \eqref{eqn: truncation estimate} is true if there exists some $\epsilon>0$ where $l_0 \ge (\log p)^{2+\epsilon}$ for all sufficiently large $p$. Without assuming GRH, for all $\epsilon>0$ there exists a constant $\kappa>0$ where Equation \eqref{eqn: truncation estimate} is true when $l_0 \ge \exp(\kappa (\log\log(p))^{5/3} (\log\log\log(p))^{1/3})$ for all sufficiently large $p$, with at most $O(p^\epsilon)$ exceptions in $t\in [-2\sqrt p,2\sqrt p]$.
\end{proposition}
The term $\exp(\kappa (\log\log(p))^{5/3} (\log\log\log(p))^{1/3})$ comes from the unconditional effective Chebotarev's theorem in \cite{PTW20}. For convenience, we will define the following.
\begin{definition}
Define $g_{\kappa}(p) \coloneqq \exp(\kappa (\log\log(p))^{5/3} (\log\log\log(p))^{1/3})$.
\end{definition}
\begin{remark}
A recent improvement by \cite{jessethorner} improves $g_\kappa(p)$ to a function that grows as a polynomial in $\log(p)$, but this does not affect our final results.
\end{remark}
\subsection{Approximating factors}
Note that $v_l(\Delta)$ is almost determined by the residue of $\Delta \mod l^{2k}$, which motivates us to define the following.
\begin{definition}\label{defn: v_l,k definition}
For a prime $l$ and any integer $k\ge 1$, define 
\begin{equation*}
v_{l,k}(\Delta) = \begin{cases} v_l(\Delta) & \text{if } \Delta \not \equiv 0 \mod l^{2k} \\
(1-l^{-2})^{-1}(1+l^{-1}) & \text{if } \Delta \equiv 0 \mod l^{2k} \end{cases}.\end{equation*}
\end{definition}
We will show that $v_{l,k}(\Delta)=v_{l,k}(t^2-4p)$ is periodic in the trace $t$ with period dividing $l^{2k}$. In the same spirit as the earlier definition of $Y_{l,p}$, we define $Y_{l,p,k}$ to be the random variable taking value $v_{l,k}(\Delta)$ when $t$ modulo $l^{2k}$ is chosen uniformly at random. 

Furthermore, $v_{l,k}(\Delta)$ is a good approximation of $v_l(\Delta)$ with error approaching zero as $k\rightarrow \infty$. Hence, if we choose $k$ to be an unbounded increasing function of $p$, we can replace the factors $v_l(\Delta)$ in \eqref{eqn: truncation estimate} with $v_{l,k}(\Delta)$ to obtain 
\begin{equation}\label{eqn: rounded off estimate}
    N_{t,p}= (1+ o(1)) \frac{\sqrt{-\Delta}}{\pi}\prod_{l<l_0}v_{l,k}(\Delta).
\end{equation}
We state this formally as follows.
\begin{proposition}\label{prop: second step}
Suppose that $k=k(p)$ is an unbounded increasing function. Assuming GRH, Equation \eqref{eqn: rounded off estimate} is true if there exists some $\epsilon>0$ where $l_0 \ge (\log p)^{2+\epsilon}$ for all sufficiently large $p$. Without assuming GRH, for all $\epsilon>0$ there exists a constant $\kappa>0$ where Equation \eqref{eqn: rounded off estimate} is true when $l_0 \ge g_\kappa (p)$ for all sufficiently large $p$, with at most $O(p^\epsilon)$ exceptions in $t\in [-2\sqrt p,2\sqrt p]$.
\end{proposition}
Since $v_{l,k}(\Delta)$ is periodic in $t$ with period dividing $l^{2k}$, the product $\prod_{l<l_0}v_{l,k}(\Delta)$ is completely periodic in trace $t$ with period dividing $T= \prod_{l<l_0} l^{2k}$. More importantly, if we pick $t$ uniformly at random in an interval of length $T$, each of the factors $v_{l,k}(\Delta)$ (which are now random variables in $t$) are independent by the Chinese Remainder Theorem. We say that these factors $v_{l,k}(\Delta)$ are independent over this interval. This resolves the second difficulty of showing independence.

We would hope that $T \ll \sqrt p$, so that the factors are almost independent in $[a\sqrt p, b\sqrt p]$ for any subinterval $[a,b]\subseteq [-2,2]$ and Theorem \ref{thm: 2D vertical sato-tate} would follow easily. Unfortunately, we compute that $T \approx e^{2kl_0}$ as the product of primes up to $N$ is approximately $e^N$, and for both values of $l_0$ (assuming GRH or without GRH) we have $T\gg \sqrt p$ which is too big. Thus, we need further insight in addition to Equation \eqref{eqn: rounded off estimate} in order to prove Theorem \ref{thm: 2D vertical sato-tate} and Theorem \ref{thm: stronger vertical sato-tate}.
\subsection{Application to 2D vertical Sato-Tate}
We wish to truncate the product further to make $T\ll \sqrt p$. We split the product $\prod_{l< l_0} v_{l,k}(\Delta)$ again into two components
$$\prod_{l< l_0} v_{l,k}(\Delta) = \prod_{l< l_1} v_{l,k}(\Delta) \prod_{l_1 \le l < l_0} v_{l,k}(\Delta).$$
where $l_1$ and $k$ are any unbounded increasing functions of $p$. The first product has period dividing $T \approx e^{2kl_1}$, and if $l_1$ and $k$ grow sufficiently slowly then we will have $T \ll \sqrt p$, so the factors in the first product are almost independent over $[a\sqrt p,b\sqrt p]$ for any subinterval $[a,b]\subseteq [-2,2]$, which is what we want.

We show the second product is almost always small. Taking the logarithm, we instead bound $\sum_{l_1 \le l< l_0} \log \left(v_{l,k}(\Delta)\right)$. We give up on having complete independence and look at pairwise independence instead. Indeed, for distinct $l,l'< l_0$, the factors $v_{l,k}(\Delta), v_{l',k}(\Delta)$ become independent over an interval of length $l^{2k}l'^{2k}\approx l_0^{4k} \ll \sqrt p$, so the two factors are almost independent over $[-2\sqrt p,2\sqrt p]$. Because the factors are almost pairwise independent, the variance of the sum $\sum_{l_1 \le l< l_0} \log \left(v_{l,k}(\Delta)\right)$ is approximately the sum of the variances $\log \left(v_{l,k}(\Delta)\right)$. Each $\log(v_{l,k}(\Delta))$ has variance $O(l^{-2})$, so adding them up shows that $\sum_{l_1 \le l< l_0} \log \left(v_{l,k}(\Delta)\right)$ has a variance of $O(l_1^{-1})=o(1)$. Applying Chebyshev's inequality tells us that $\sum_{l_1 \le l< l_0} \log \left(v_{l,k}(\Delta)\right)=o(1)$ for almost all $t$. This almost-all bound is sufficient to show convergence in distribution, which is good enough for proving Theorem \ref{thm: 2D vertical sato-tate}. Thus, we obtain:
\begin{proposition}\label{prop: chebyshev result}
Suppose $l_1=l_1(p)$ and $k=k(p)$ are unbounded increasing functions, and suppose additionally that there exists some $\epsilon>0$ where $k=O((\log p)^{1-\epsilon})$. Without assuming GRH, we have
$$
    N_{t,p}= (1+ o(1)) \frac{\sqrt{-\Delta}}{\pi}\prod_{l<l_1}v_{l,k}(\Delta)$$
with at most $o(\sqrt p)$ exceptions in $t\in [-2\sqrt p, 2\sqrt p]$.
\end{proposition}
Lastly, we will apply this proposition to prove Theorem \ref{thm: 2D vertical sato-tate}. Let $T=\prod_{l<l_1}l^{2k}\approx e^{2kl_1}$ so that the period of the product above divides $T$, and choose $l_1$ and $k$ such that $T=o(\sqrt p)$. For any subinterval $[t_0,t_0+T)\subseteq [-2\sqrt p, 2\sqrt p]$ of length $T$, pick a random trace $t\in [t_0,t_0+T)$. Then, in this subinterval, the following distributions converge as $p\rightarrow \infty$:
$$\frac{N_{t,p}}{2\sqrt p}\sim \frac{\sqrt{4-x^2}}{2\pi}\prod_{l<l_1}v_{l,k}(\Delta) = \frac{\sqrt{4-x^2}}{2\pi}\prod_{l<l_1}Y_{l,p,k}\sim \frac{\sqrt{4-x^2}}{2\pi} Z_p$$
where we take the independent product of $Y_{l,p,k}$ (defined below Definition \ref{defn: v_l,k definition}) and $x=t/\sqrt p$ is the normalized trace. The first approximation is due to Proposition \ref{prop: chebyshev result}. The middle equality is because the factors $v_{l,k}(\Delta)$ are independent in an interval of length $T$ by the Chinese Remainder Theorem. Finally, the last approximation follows from the fact that $Y_{l,p,k}$ are good approximations of $Y_{l,p}$ and that the tail end of the infinite product $Z_p$, which is $\prod_{l\ge l_1} Y_{l,p}$, is almost always small. Essentially, what we did was to truncate and approximate Gekeler's formula to show independence, after which we restored the original form with independence. To complete the proof of Theorem \ref{thm: 2D vertical sato-tate}, we break $[-2\sqrt p, 2\sqrt p]$ into intervals of length $T$ and apply the approximation that we obtained above for each interval.
\subsection{Application to stronger vertical Sato-Tate}\label{sec: proof overview strong}
To prove Theorem \ref{thm: stronger vertical sato-tate}, we want to show that the quantity
$$\frac{\#\{E/\F_p \mid t_p(E)\in \left[t_0,t_0+f(p)\right)\}}{\frac{f(p)}{\sqrt p}\#\{E/\F_p\}} \approx \sqrt{4-(t_0/\sqrt p)^2}.$$
We understand the denominator using $\#\{E/\F_p\} = 2p + O(1)$. On the other hand, we can express the numerator as
$$\#\left\{E/\F_p \mid t_p(E)\in \left[t_0,t_0+f(p)\right) \right\} = \sum_{t=t_0}^{t_0+f(p)-1} N'_{t,p} \approx \sum_{t=t_0}^{t_0+f(p)-1} N_{t,p}$$
By Proposition \ref{prop: second step}, we have the following approximation
$$N_{t,p} \approx \frac{\sqrt{-\Delta}}{\pi}\prod_{l<l_0}v_{l,k}(\Delta),$$
where if we assume GRH this approximation applies for all traces, but without assuming GRH there may be $O(p^\epsilon)$ exceptional values of $t$. For these exceptional values we will use a much weaker bound for $N_{t,p}$. This is the reason why we require $f(p) \ge p^\epsilon$ in the unconditional statement of Theorem \ref{thm: stronger vertical sato-tate}, so that the approximation above holds for most $t\in [t_0,t_0+f(p))$.

After accounting for the exceptions, the approximation above allows us to reduce the problem to estimating $\E[\prod_{l<l_0}v_{l,k}(\Delta)]$ for a uniform random integer $t\in [t_0,t_0+f(p))$ chosen in the interval. We write
\begin{equation}\label{eqn: taylor expansion}
\prod_{l<l_0}v_{l,k}(\Delta) = \exp\left(\sum_{l<l_0} \log(v_{l,k}(\Delta))\right) = \sum_{n=0}^{\infty} \frac{1}{n!}\left(\sum_{l<l_0} \log(v_{l,k}(\Delta))\right)^n.
\end{equation}
It suffices to estimate the expectation $\E[\left(\sum_{l<l_0} \log(v_{l,k}(\Delta))\right)^n]$ for all reasonably small $n$, as for large $n$ the $n!$ term dominates. Expanding the power and using linearity of expectation, we see that $\E[\left(\sum_{l<l_0} \log(v_{l,k}(\Delta))\right)^n]$ is made up of terms of the form $\E[\prod_i \log(v_{l_i,k}(\Delta))^{\alpha_i}]$ with $\sum_i \alpha_i = n$. Now we recall that $\log(v_{l_i,k}(\Delta))$ is periodic in $t$ with period dividing $l_i^{2k}<l_0^{2k}$, so for reasonably small $n$ the terms $\log(v_{l_i,k}(\Delta))^{\alpha_i}$ are independent in an interval of period dividing $T\approx l_0^{2kn} \ll f(p)$, so they are almost independent in the interval $[t_0,t_0+f(p))$. This allows us to estimate $\E[\prod_{i} \log(v_{l_i,k}(\Delta))^{\alpha_i}]$, which is what we wanted. 
\section{Truncation formulas}\label{sec: truncation formula}
\subsection{Effective Chebotarev's density theorems} 
We state two effective Chebotarev's density theorems, one assuming GRH and one without. We will specialize to the case of quadratic extensions over $\Q$, and state both theorems in this special case. 

\begin{definition} Let $K/\Q$ be a quadratic extension.  Denote the counting functions
\begin{equation*}
\begin{split}
\pi_{+}(x,K) &\coloneqq \# \{\text{primes } p\le x \text{ inert in } K\},\\
\pi_{-}(x,K) &\coloneqq \# \{\text{primes } p\le x \text{ that split in } K\}.
\end{split}
\end{equation*}
Note that both functions do not count ramified primes. 
\end{definition}
\begin{theorem}[\cite{LO75}, Theorem 1.1]\label{thm: eff grh}
Assuming GRH, there exists an effectively computable absolute constant $C>0$ such that for a quadratic extension $K/\Q$ with discriminant $\Delta_0$,
\begin{equation*}
\left| \pi_{\pm}(x,K) - \frac 1 2 \Li(x) \right| \le C x^{1/2} \log(|\Delta_0| x).
\end{equation*}
\end{theorem}
\begin{definition}
Let $\mathscr F$ be the family of all quadratic extensions $K/\Q$, and $\mathscr F(X)$ be the family of all quadratic extensions $K/\Q$ with absolute value of the discriminant being at most $X$.
\end{definition}
The following theorem does not assume GRH.
\begin{theorem}[\cite{PTW20}, Theorem 3.3, Case 1]\label{thm: eff no grh} Fix any $A\ge 2$ and any sufficiently small $\epsilon > 0$. Then, there exists constants $C,\kappa>0$ and a subfamily of exceptions $\mathscr E \subseteq \mathscr F$, such that for all $X\ge 1$, $\mathscr E(X) := \mathscr E \cap \mathscr F(X)$ has size at most $CX^\epsilon$.

Furthermore, each $K\in \mathscr F \setminus \mathscr E$ satisfies
\begin{equation*}
\left|\pi_\pm (x,K) - \frac 1 2 \Li(x)\right| \le \frac 1 2 \frac{x}{(\log x)^A}
\end{equation*}
for all $x\ge g_{\kappa}(|\Delta_0|)$ where $\Delta_0$ is the discriminant of $K$.
\end{theorem}
Recall that we want to bound $\prod_{l>l_0} v_l(\Delta)$, and that we defined $\Delta = f^2 \Delta_0$ where $\Delta_0$ is the fundamental discriminant. By Lemma \ref{lem: v_l facts} we can write
\begin{equation*}
\prod_{l\ge l_0} v_l(\Delta) = \prod_{l\ge l_0} \left(1 + \frac{a_l}{l}\right)^{-1}
\end{equation*}
where $a_l\in [-1,1]$, and if $l\nmid \Delta$ then $a_l = -\left(\frac{\Delta}{l}\right) = -\left(\frac{\Delta_0}{l}\right)$. Note that if $l$ is inert in $K$, which means that $a_l = -(\frac{\Delta_0}{l})=+1$. Likewise, if $l\in \pi_- (K)$ then $a_l=-1$. Hence, if the conditions for $K$ in either Theorem \ref{thm: eff grh} or \ref{thm: eff no grh} hold, then both $+1$ and $-1$ appear about half the time with concrete bounds.
\subsection{Dyadic intervals}
The following combinatorial lemma bounds the product over factors $(1+a_l/l)$ dyadic intervals when we know that $a_l$ takes the value $+1$ and $-1$ each about half the time.
\begin{lemma}\label{lem: dyadic bound}
Let $a_l\in [-1,1]$ be a sequence of reals indexed by primes $l$, and let $n$ be a positive integer. Suppose that for all $2^n\le x\le 2^{n+1}$ that 
\begin{equation}\label{eqn: bounds dyadic lemma}
\begin{split}
\left| \# \{\text{primes }l < x \mid a_l = +1\} - \frac 1 2 \# \{\text{primes }l < x\} \right| &\le C\\
\left| \# \{\text{primes }l < x \mid a_l = -1\}- \frac 1 2 \# \{\text{primes }l < x\} \right| &\le C
\end{split}
\end{equation}
for some upper bound $C$. Then, we have
\begin{equation*}
\log\left(\prod_{2^n\le l < 2^{n+1}} \left (1+\frac{a_l}{l}\right)\right) = O(C2^{-n}).
\end{equation*}
\end{lemma}
\begin{proof}
First we bound the terms where $a_l \not \in \{\pm 1\}$. Indeed, by taking $x=2^{n+1}$ in Equation \eqref{eqn: bounds dyadic lemma} we have that $\# \{\text{primes } l < 2^{n+1} \mid a_l \not \in \{\pm 1\} \} \le 2C$. Since each term $\log(1+a_l/l) = O(l^{-1}) = O(2^{-n})$, in total these terms contribute $O(C2^{-n})$ to the product.

Now we deal with the terms where $a_l\in \{\pm 1\}$. Label the primes $l \in [2^n,2^{n+1})$ with $a_l=+1$ as $p_1,p_2,\ldots ,p_k$ and label those with $a_l=-1$ as $q_1,q_2,\ldots ,q_m$, both in increasing order. 

We claim that $p_i<q_{i+4C}$ for all $i$. Suppose otherwise that $q_{i+4C}<p_i$, then 
\begin{equation*}
\left|\#\{\text{primes } 2^n\le l < p_i \mid a_l = +1 \} - \#\{\text{primes } 2^n\le l < p_i  \mid a_l = -1 \}\right| \ge i+4C -(i-1) > 4C.
\end{equation*}
because the first contains exactly $i-1$ terms $p_1,\ldots, p_{i-1}$ while the second term must contain at least $i+4C$ terms $q_1,\ldots,q_{i+4C}$.
But on the other hand we have
\begin{equation*}
\begin{split}
&  \left|\#\{\text{primes } 2^n\le l < p_i \mid a_l = +1 \} - \#\{\text{primes } 2^n\le l < p_i  \mid a_l = -1 \}\right| \\
\le &  \left|\#\{\text{primes } l < p_i \mid a_l = +1 \} - \#\{\text{primes } l < p_i  \mid a_l = -1 \}\right|  \\
&+ \left|\#\{\text{primes } l<2^n \mid a_l = +1 \} - \#\{\text{primes } l < 2^n \mid a_l = -1 \}\right|  \\
\le & \  2C + 2C = 4C
\end{split}
\end{equation*}
where we used the triangle inequality together with the substitutions $x=2^n$ and $x=p_i$ in Equation \eqref{eqn: bounds dyadic lemma}. The two equations above yield a contradiction, so we indeed have $p_i<q_{i+4C}$.

Next, we show a lower bound 
\begin{equation}\label{eqn: dyadic pf 3}
    \log \left(\prod_i (1+\frac{1}{p_i}) \prod_i (1-\frac{1}{q_i}) \right)\ge -AC2^{-n}
\end{equation}
where $A>0$ is a universal constant independent of $C$ and $n$.

Pair up the primes $(p_1,q_{4C+1}), (p_2,q_{4C+2}),\ldots (p_{m-4C},q_{m})$, with some primes remaining. Then, for each pair we have $(1+1/p_i)(1-1/q_{4C+i})>(1+1/p_i)(1-1/p_i) = (1-1/p_i^2)$. There exists some universal constant $B$ where $\log(1-1/n)\ge -B/n$ for all integers $n\ge 2$, so we have
$$\log\left(\left(1+\frac{1}{p_i}\right)\left(1-\frac{1}{q_{4C+i}}\right)\right)\ge -\frac{B}{p_i^2} \ge -B4^{-n}.$$ There are at most $2^{n-1}$ such pairs, so in total the contribution of the paired primes to right side of Equation \eqref{eqn: dyadic pf 3} is at least $-B2^{-n-1}$. 

The remaining primes are $q_1,\ldots ,q_{4C}$ and $p_{m-4C+1},\ldots ,p_{k}$. Each $\log(1+1/p_i)$ is non-negative, and each $\log(1-1/q_i)$ is at least $-B/q_i\ge -B2^{-n}$ each, so the log of the terms for the remaining primes is at least $-4BC2^{-n}$. Combining the contributions from the paired primes and the unpaired primes, we obtain the Equation \eqref{eqn: dyadic pf 3}.

Repeating the same argument with $+$ and $-$ flipped, we get an upper bound of the same magnitude. Hence, $\log \left(\prod_i (1+\frac{1}{p_i}) \prod_i (1-\frac{1}{q_i}) \right)=O(C2^{-n})$. Combining with the terms where $a_l\not \in \{\pm 1\}$ gives the result.
\end{proof}
\begin{remark}\label{rem: dyadic remark}
Instead of having the product of primes from $2^n$ to $2^{n+1}$, the above proof also works for subintervals of it. In our particular case, suppose that $l_0$ satisfies $2^n\le l_0 < 2^{n+1}$ and the same conditions in Lemma \ref{lem: dyadic bound} hold. Then we would also have
\begin{equation*}
\log\left(\prod_{l_0\le l < 2^{n+1}} \left (1+\frac{a_l}{l}\right)\right) = O(C2^{-n}).
\end{equation*}
\end{remark}
\subsection{Tail bounds}
Now we apply Lemma \ref{lem: dyadic bound} to Theorem \ref{thm: eff grh} and \ref{thm: eff no grh} respectively to obtain the following bounds. 
\begin{proposition}\label{prop: GRH tail bound}
Assuming GRH, we have
\begin{equation*}\label{eqn: l>l_0 bounds grh}
\log\left(\prod_{l\ge l_0} v_l(\Delta) \right) = O\left(l_0^{-1/2}\log(|\Delta_0| l_0)\right).
\end{equation*}
for all $l_0$.
\end{proposition}
\begin{proof}
Assuming the Riemann Hypothesis, it is a long standing result by \cite{von_koch_helge_1901_2347595} that the number of primes less than $x$, which we denote as $\pi(x)$, satisfies $\pi(x) = \Li(x) + O(x^{1/2}\log(x))$. Combining this with Theorem \ref{thm: eff grh}, we obtain the bound $|\pi_\pm (x,K) - \frac 1 2 \pi(x)| = O(x^{1/2}\log(|\Delta_0| x))$. Thus we get $|\pi_\pm (x,K) - \frac 1 2 \pi(x)| = O(2^{n/2}(\log|\Delta_0|+n))$ for $2^n\le x\le 2^{n+1}$. Taking this to be $C$ in Lemma \ref{lem: dyadic bound}, we have
\begin{equation*}
\log\left(\prod_{2^n\le l < 2^{n+1}} v_l(\Delta)\right) = O\left(2^{-n/2}(\log|\Delta_0|+n)\right).
\end{equation*}
Let $n$ be such that $2^n\le l_0<2^{n+1}$. Take the sum of this on the intervals $[l_0,2^{n+1}],[2^{n+1},2^{n+2}],\ldots,$ while keeping in mind that the same bound applies to the first interval by Remark \ref{rem: dyadic remark}. The first term dominates, so we obtain
\begin{equation*}
\log\left(\prod_{l\ge l_0} v_l(\Delta)\right) = O\left(2^{-n/2}(\log|\Delta_0|+n)\right) = O\left(l_0^{-1/2}\log(|\Delta_0| l_0)\right).
\end{equation*}
\end{proof}
\begin{proposition}\label{prop: no GRH tail bound}
Without assuming GRH, fix any $A\ge 2$ and any sufficiently small $\epsilon>0$, then there exists a constant $\kappa>0$ such that for almost all traces $t\in [-2\sqrt{p},2\sqrt{p}]$ aside from $O(p^\epsilon)$ many exceptions, we have
\begin{equation*}\label{eqn: high l bound}
\log\left(\prod_{l\ge l_0}v_l(\Delta)\right) = O((\log l_0)^{1-A})
\end{equation*}
for $l_0\ge g_{\kappa}(p)$.
\end{proposition}
\begin{proof}
First suppose that $K=\Q(\Delta_0)$ is not in the subfamily of exceptions $\mathscr E$. In this case, we follow the proof of the previous proposition. \cite{charles} proved that $\pi(x) = \Li(x) + O(xe^{-a\sqrt{\log x}})$, so combining with the error in Theorem \ref{thm: eff no grh} gives $\left|\pi_\pm (x,K) - \frac 1 2 \pi(x)\right| \le O(\frac{x}{(\log x)^A})$ for all $x\ge g_\kappa (|\Delta_0|)$. Thus, Lemma \ref{lem: dyadic bound} gives 
\begin{equation*}
\log\left(\prod_{2^n\le l < 2^{n+1}} v_l(\Delta)\right) = O(n^{-A})
\end{equation*}
for $2^n\ge g_{\kappa}(|\Delta_0|)$.

For any $n$ where $2^n\ge g_{\kappa}(|\Delta_0|)$, and for any $2^n\le l_0<2^{n+1}$, take the sum of this on the dyadic intervals $[l_0,2^{n+1}],[2^{n+1},2^{n+2}],\ldots$. Because $n^{-A}+(n+1)^{-A}+\cdots = O(n^{1-A})$ we get
\begin{equation*}
\log\left(\prod_{l\ge l_0} v_l(\Delta)\right) = O(n^{1-A}) = O((\log l_0)^{1-A}).
\end{equation*}
We see that this bounds applies to all $l_0 \ge 2g_{\kappa}(|\Delta_0|)$. Since $|\Delta_0|\le |\Delta| = |t^2-4p| \le 4p$, this applies to all $l_0\ge g_{\kappa'}(p)$ where we replaced $\kappa$ with a larger constant $\kappa'$.

We are left to bound the number of exceptional traces $t$ where $K=\Q(\Delta)=\Q(t^2-4p)\in \mathscr E$ is in the subfamily of exceptions. Consider such a trace $t$. The discriminant of $K$ is the fundamental discriminant $\Delta_0$ with $|\Delta_0|\le 4p$, so $K\in \mathscr E(4p)$. By Theorem \ref{thm: eff no grh}, $\mathscr E(4p)$ has size at most $O(p^\epsilon)$. Thus, it suffices to prove that for each $K\in \mathscr E$, there are only $O(1)$ number of $t$ where $K=\Q(t^2-4p)$. Since $t^2-4p = \Delta = f^2\Delta_0$, it suffices to bound the number of solutions to $4p=t^2-f^2\Delta_0$, which is at most the number of elements $a=t+f\sqrt{\Delta_0}\in K$ which have norm $4p$. In other words, we have $(a)(\overline a)=(4p)$ as ideals. Write $(4p) = \prod_{j\in J} \mathcal I_j$ as product of not necessarily distinct prime ideals, so taking the norm on both sides gives $16p^2 = \prod_{j\in J} N(\mathcal I_j)$. Since each $N(\mathcal I_j)>1$, we can have at most $6$ not necessarily distinct prime ideals. Then we must have $(a) = \prod_{j\subseteq J} I_j$ and there are only $O(1)$ such combinations.
\end{proof}
Now Proposition \ref{prop: first step} follows easily from these bounds.
\begin{proof}[Proof of Proposition \ref{prop: first step}]
Assuming GRH, substituting $l_0\ge (\log p)^{2+\epsilon}$ into Proposition \ref{prop: GRH tail bound} gives $\log\left(\prod_{l\ge l_0} v_l(\Delta) \right) = o(1)$ since $|\Delta_0| = O(p)$, so $\prod_{l\ge l_0} v_l(\Delta) = 1 + o(1)$ and Equation \eqref{eqn: truncation estimate} holds. Similarly, without assuming GRH, substituting $l_0\ge g_{\kappa}(p)$ into Proposition \ref{prop: no GRH tail bound} also yields Equation \eqref{eqn: truncation estimate} with $O(p^\epsilon)$ exceptional values of $t$.
\end{proof}
\subsection{Approximating factors} Recall the definition of $v_{l,k}$ in Definition \ref{defn: v_l,k definition}. First, we show that $v_{l,k}(\Delta)$ are periodic in the trace $t$ and are good approximations of $v_l(\Delta)$.
\begin{lemma}\label{lem: basic v_l,k facts}
The factor $v_{l.k}(\Delta) = v_{l,k}(t^2-4p)$ is periodic in $t$ with period dividing $l^{2k}$. Furthermore, $\log(v_{l,k}(\Delta)) = \log(v_l(\Delta))+ O(l^{-k})$.
\end{lemma}
\begin{proof}
It is easy to check both statements. For the first it suffices to show that $v_{l,k}(\Delta) = v_{l,k}(\Delta + l^{2k}) \mod {l^{2k}}$. This is clear when $l^{2k}\mid \Delta$ by the definition of $v_{l,k}$. Otherwise, when $l^{2k}\nmid \Delta$, we want to show that $v_l(\Delta) = v_l (\Delta + l^{2k})$. This is easy to see from Equation \eqref{eqn: v_l original}, as both $\delta$s are the same for both $\Delta$ and $\Delta + l^{2k}$, and we also have $\left(\frac{\Delta/l^{2\delta}}{l}\right) = \left(\frac{(\Delta+l^{2k})/l^{2\delta}}{l}\right)$.

For the second statement, by the definition above, $v_{l,k}(\Delta)\neq v_l(\Delta)$ only when $\Delta \equiv 0\mod l^{2k}$. In this case, by Equation \eqref{eqn: v_l original} we see that $\delta \ge k-1$ and so the error is $O(l^{-k})$.
\end{proof}
Because $v_{l,k}(\Delta)$ are good approximations of $v_l(\Delta)$, we replace the factors $v_l(\Delta)$ in Proposition \ref{prop: first step} with $v_{l,k}(\Delta)$ to obtain Proposition \ref{prop: second step} as follows.
\begin{proof}[Proof of Proposition \ref{prop: second step}]
Note that for $k>1$,
\begin{equation}\label{eqn: prod v_l,k approximate prod v_l}
\log \left(\prod_{l<l_0} v_{l,k}(\Delta)\right) = \log \left(\prod_{l<l_0} v_{l}(\Delta)\right) + O(2^{-k})
\end{equation}
by summing up the second statement of Lemma \ref{lem: basic v_l,k facts} over $l< l_0$, and noting that $2^{-k}+3^{-k}+\ldots = O(2^{-k})$ for $k>1$. As $k$ is an unbounded increasing function we have $O(2^{-k})=o(1)$, so $$\prod_{l<l_0} v_{l}(\Delta) = (1+o(1))\prod_{l<l_0} v_{l,k}(\Delta).$$ The proposition then follows directly from Proposition \ref{prop: first step} by substituting the above equation into Equation \eqref{eqn: truncation estimate} to obtain Equation \eqref{eqn: rounded off estimate}.
\end{proof}
The periodicity of the factors $v_{l,k}(\Delta)$ in the trace $t$ leads to independence of the factors $v_{l,k}(\Delta)$ over a large enough interval for the trace $t$.
\begin{proposition}\label{prop: period is T}
The product
$$\prod_{l<l_0} v_{l,k}(\Delta) = \prod_{l<l_0} v_{l,k}(t^2-4p)$$
is periodic in $t$ with period dividing $T = \left(\prod_{l< l_0} l\right)^{2k} = O(2^{4kl_0})$. Furthermore, the factors $\{v_{l,k}(\Delta)\}_{l<l_0}$ are independent in any interval of length $T$.
\end{proposition}
\begin{proof}
This follows directly from the first part of Lemma \ref{lem: basic v_l,k facts} and by the Chinese Remainder Theorem.
\end{proof}
Unfortunately, as explained in the proof overview, given the values of $l_0$ in Proposition \ref{prop: second step}, $T$ is much greater than $\sqrt p$, so we cannot apply this directly.
\section{2D vertical Sato-Tate}\label{sec: pf of 2D sato-tate}
We will not require GRH for the results we derive in this section. Hence, for this section, we fix any $\epsilon>0$ and set $l_0=g_\kappa (p)$ for some $\kappa$ such that the unconditional statement in Proposition \ref{prop: second step} is true.
\subsection{Further truncation}
Recall that the truncation we obtained in the previous section was not enough for complete independence. We will use pairwise independence to further truncate the product. Notice that any two variables $v_{l,k}$ and $v_{l',k}$ are independent over an interval $(ll')^{2k}\le l_0^{4k} \ll p$. This leads to almost independence over $t\in [-2\sqrt{p},2\sqrt{p}]$. We quantify this in the following lemma.
\begin{lemma}\label{lem: individual var and covar}
Suppose we choose $t\in [-2\sqrt{p},2\sqrt{p}]$ uniformly at random. Then, for any prime $l$ we have
\begin{equation*}
\Var\left[\log(v_{l,k}(\Delta))\right] = O(l^{-2}).
\end{equation*}
Furthermore, for distinct primes $l,l' <l_0$ we have
\begin{equation*}
\Cov\left(\log(v_{l,k}(\Delta)), \log(v_{l',k}(\Delta))\right) = O(l_0^{4k-2}p^{-1/2}).
\end{equation*}
\end{lemma}
\begin{proof}
The first statement is clear by the definition of $v_{l,k}$ and the fact that $\log(v_l(\Delta))=O(l^{-1})$ by Lemma \ref{lem: v_l facts}(b). For the second statement,
by the first part of Lemma \ref{lem: basic v_l,k facts} and the Chinese Remainder Theorem, $v_{l,k}(\Delta)$ and $v_{l',k}(\Delta)$ are independent over any interval of length $T=(ll')^{2k}$. Split the interval $t\in [-2\sqrt{p},2\sqrt{p}]$ into intervals of length $T$, with one interval $I$ of length $0\le L <T$ remaining. Note that
$$\Cov\left(\log(v_{l,k}(\Delta)), \log(v_{l',k}(\Delta))\right) = \E[\log(v_{l,k}(\Delta))\log(v_{l',k}(\Delta))-\E[\log(v_{l,k}(\Delta)]\E[\log(v_{l',k}(\Delta))]].$$
On the intervals with length $T$, $v_{l,k}(\Delta)$ and $v_{l',k}(\Delta)$ are independent, which implies that the average of $\log(v_{l,k}(\Delta))\log(v_{l',k}(\Delta))-\E[\log(v_{l,k}(\Delta)]\E[\log(v_{l',k}(\Delta))]$ is $0$ on the interval. Thus, the only contribution to the covariance is when $t\in I$, which happens with probability $O((ll')^{2k}p^{-1/2})$. Furthermore, by definition of $v_{l,k}$, both $\log(v_{l,k}(\Delta))$ and $\log(v_{l',k}(\Delta))$ are $O(l^{-1})$ and $O(l'^{-1})$ respectively. Combining both facts gives $$\Cov\left(\log(v_{l,k}(\Delta)), \log(v_{l',k}(\Delta))\right) = O((ll')^{2k-1}p^{-1/2}) = O(l_0^{4k-2}p^{-1/2}).$$
\end{proof}
Next, we use each individual variance and covariance to bound the variance of the following sum.
\begin{proposition}\label{prop: variance sum}
Suppose we choose $t\in [-2\sqrt{p},2\sqrt{p}]$ uniformly at random. Then
\begin{equation*}
\Var\left[\log\left(\prod_{l_1\le l < l_0} v_{l,k}(\Delta) \right)\right] = O(l_1^{-1}+l_0^{4k}p^{-1/2}).
\end{equation*}
\end{proposition}
\begin{proof}
We have
\begin{equation*}
\begin{split}
&\ \Var\left[\log\left(\prod_{l_1\le l < l_0} v_{l,k}(\Delta) \right)\right] \\
=&\ \sum_{l_1\le l \le l_0} \Var[\log(v_{l,k}(\Delta))] +  2\sum_{l_1\le l < l' \le l_0} \Cov\left(\log(v_{l,k}(\Delta)), \log(v_{l',k}(\Delta))\right) 
\end{split}
\end{equation*}
where by Lemma \ref{lem: individual var and covar} the first sum is $O(l_1^{-1})$ as $l_1^{-2}+(l_1+1)^{-2}+\cdots = O(l_1^{-1})$, and the second sum has $O(l_0^2)$ terms each of size $O(l_0^{4k-2}p^{-1/2})$.
\end{proof}
\begin{remark}
We note that the initial truncation using the effective Chebotarev's density theorems was essential for this step as the variance has an $l_0^{4k}$ term in it. Intuitively, this is because if we did not truncate the infinite product, we would not have pairwise independence between $v_{l,k}(\Delta)$ and $v_{l',k}(\Delta)$ if $l$ and $l'$ are too big.
\end{remark}
We apply Chebyshev's inequality to prove Proposition \ref{prop: chebyshev result}.
\begin{proof}[Proof of Proposition \ref{prop: chebyshev result}] 
Define $v(p)=\Var\left[\log\left(\prod_{l_1\le l < l_0} v_{l,k}(\Delta) \right)\right]$. Since $k = O((\log p)^{1-\epsilon})$, $l_0 = g_{\kappa}(p)$ and $l_1$ is an unbounded increasing function of $p$, by substitution we see that $l_1^{-1}+l_0^{4k}p^{-1/2}=o(1)$, so Proposition \ref{prop: variance sum} tells us that $v(p)=o(1)$. By Chebyshev's inequality,
\begin{equation*}
\mathbb P \left[\log\left(\prod_{l_1\le l < l_0} v_{l,k}(\Delta) \right) \ge v(p)^{1/3} \right] \le v(p)^{1/3}.
\end{equation*}
Since $v(p)^{1/3} = o(1)$, this says that for almost all $t\in [-2\sqrt p, 2\sqrt p]$ we have $\log\left(\prod_{l_1\le l < l_0} v_{l,k}(\Delta) \right) = o(1)$ and thus $$\prod_{l_1\le l < l_0} v_{l,k}(\Delta) = 1+o(1).$$ By the unconditional statement of Proposition \ref{prop: second step}, we substitute the above equation into Equation \eqref{eqn: rounded off estimate} to obtain
$$N_{t,p}= (1+ o(1)) \frac{\sqrt{-\Delta}}{\pi}\prod_{l<l_0}v_{l,k}(\Delta) = (1+ o(1)) \frac{\sqrt{-\Delta}}{\pi}\prod_{l<l_1}v_{l,k}(\Delta)$$
for almost all $t\in [-2\sqrt p, 2\sqrt p]$.
\end{proof}
\subsection{Proof of 2D vertical Sato-Tate}
Fix $l_1$ and $k$ to satisfy the conditions in Proposition \ref{prop: chebyshev result} and let $T\coloneqq \prod_{l<l_1}l^{2k} = o(\sqrt p)$. Let the number of exceptions in Proposition \ref{prop: chebyshev result} be $e(p)=o(\sqrt p)$. For each $p$, choose $T'$ to be a multiple of $T$, and do this such that the growth rate of $T'$ is between $e(p)$ and $\sqrt p$, i.e. $e(p) = o(T')$ and $T' = o(\sqrt p)$. We will use this as our new period.

We wish to prove that $(t/\sqrt p, N_{t,p}/2\sqrt p)\sim (X,\frac{1}{2\pi}\sqrt{4-X^2}Z_p)$ converges in the Lévy-Prokhorov metric. To do this, we first look at both distributions on small vertical slices $[t_0/\sqrt p, (t_0+T')/\sqrt p)\times [0,\infty)$. On such a small vertical slice, we prove that the second coordinates of both distributions, $N_{t,p}/2\sqrt p$ and $\frac{1}{2\pi}\sqrt{4-X^2}Z_p$, get closer as $p\rightarrow \infty$.
\begin{proposition}\label{prop: small vertical slices}
For a subinterval $[t_0,t_0+T')\subseteq [-2\sqrt p, 2\sqrt p]$ of length $T'$, let $\alpha_{p,t_0}$ be the distribution of $N_{t,p}/2\sqrt p$ where the trace $t$ is chosen uniformly at random in $[t_0,t_0+T')$. Also, let $\alpha_{p,t_0}^{heu}$ be the distribution of $\frac{1}{2\pi}\sqrt{4-X^2}Z_p$ as an independent product where $X$ is the uniform real random variable in $\left[t_0/\sqrt p,(t_0+T')/\sqrt p\right)$. Then, we have
$$\sup_{t_0}\left(\pi(\alpha_{p,t_0},\alpha_{p,t_0}^{heu})\right)\rightarrow 0$$
as $p\rightarrow \infty$.
\end{proposition}
\begin{proof}
Define the intermediate distribution $\alpha_{p,t_0}'$ to be the distribution of the random variable $$\frac{\sqrt{4-(t_0/\sqrt p)^2}}{2\pi}  \prod_{l<l_1}Y_{l,p,k},$$
where we take the independent product of $Y_{l,p,k}$. 

Note that $\alpha_{p,t_0}'$ is also the distribution of $\frac{\sqrt{4-(t_0/\sqrt p)^2}}{2\pi}  \prod_{l<l_1}v_{l,k}(\Delta)$ where $t$ is chosen uniformly at random in $[t_0,t_0+T')$. This is because the factors $v_{l,k}(\Delta)$ are independent over any interval of length $T$ by Proposition \ref{prop: period is T}, and hence are independent over any interval of length $T'$ as well. Now, it suffices to show that both $\pi(\alpha_{p,t_0},\alpha_{p,t_0}') \rightarrow 0$ and $\pi(\alpha_{p,t_0}',\alpha_{p,t_0}^{heu}) \rightarrow 0$ uniformly in $t_0$ as $p\rightarrow \infty$.

For the first, note that by Proposition \ref{prop: chebyshev result} we have 
\begin{equation}\label{eqn: pf 2D eqn 1}
N_{t,p}= (1+ o(1)) \frac{\sqrt{4-(t/\sqrt p)^2}}{2\pi}\prod_{l<l_1}v_{l,k}(\Delta)
\end{equation}
with at most $e(p)=o(T')$ exceptions by our choice of $T'$. Furthermore, since $|t/\sqrt p-t_0/\sqrt p|\le T'/\sqrt p = o(1)$, we get 
\begin{equation}\label{eqn: pf 2D eqn 2}
\sqrt{4-\left(\frac{t}{\sqrt p}\right)^2} = \sqrt{4-\left(\frac{t_0}{\sqrt p}\right)^2} + o(1),
\end{equation}
and this bound is clearly uniform in $t_0$. Since $v_{l,k}(\Delta)$ is independent in the interval $[t_0,t_0+T')$, we have
$$\Var\left[\sum_{l<l_1} \log(v_{l,k}(\Delta))\right] = \sum_{l<l_1} \Var[\log(v_{l,k}(\Delta))].$$ 
By definition of $v_{l,k}$, we have $\Var[\log(v_{l,k}(\Delta))] = O(l^{-2})$. Furthermore, since $\sum_{l} l^{-2}$ converges, by the above equation there exists a universal constant $C$ independent of $p$ where $\Var[\sum_{l<l_1} \log(v_{l,k}(\Delta))] \le C$. Then, for any increasing unbounded function $f_1(p)$, we have by Chebyshev's inequality that almost all $t\in [t_0,t_0+T')$ (with at most $o(T')$ exceptional values of $t$) satisfy
$$\sum_{l<l_1} \log(v_{l,k}(\Delta)) \le f_1(p),$$
thus implying that for any increasing unbounded function $f_2(p)$ that almost all $t\in [t_0,t_0+T')$ satisfy
\begin{equation}\label{eqn: pf 2D eqn 3}
\prod_{l<l_1} v_{l,k}(\Delta) \le f_2(p).
\end{equation}
Combining Equations \eqref{eqn: pf 2D eqn 1}, \eqref{eqn: pf 2D eqn 2} and \eqref{eqn: pf 2D eqn 3}, we obtain that for almost all $t\in [t_0,t_0+T')$ that
\begin{equation*}
\begin{split}
N_{t,p} &= (1+o(1))\left(\frac{\sqrt{4-(t_0/\sqrt p)^2}+o(1)}{2\pi}\right)\prod_{l<l_1}v_{l,k}(\Delta) \\
&= \frac{\sqrt{4-(t_0/\sqrt p)^2}}{2\pi}\prod_{l<l_1}v_{l,k}(\Delta) + o(1) \\
\end{split}
\end{equation*}
by choosing the function $f_2(p)$ to grow slow enough such that both $o(1)$ terms decrease faster than $\frac{1}{f_2(p)}$. Rephrasing this, for any $\epsilon$, we have for all sufficiently large $p$ that $|\alpha_{p,t_0}-\alpha_{p,t_0}'|<\epsilon$ with probability greater than $1-\epsilon$. Thus, by Lemma \ref{lem: bound metric} we conclude that $\pi(\alpha_{p,t_0},\alpha_{p,t_0}') \rightarrow 0$ uniformly in $t_0$ as $p\rightarrow \infty$.

Now we show that $\pi(\alpha_{p,t_0}',\alpha_{p,t_0}^{heu}) \rightarrow 0$ uniformly in $t_0$ as $p\rightarrow \infty$. For this proof, recall that $Y_{l,p,k}$ is dependent on $Y_{l,p}$ for each $l$ and $p$ as follows as both share the same event space which is $\Z_l$. More explicitly, $Y_{l,p}$ is the random variable taking values $v_l(t^2-4p)$ where $t\in \Z_l$ is picked uniformly according to the Haar measure. Using the same random $t$, let $Y_{l,p,k}$ be the random variable taking values $v_{l,k}(t^2-4p)$. This is well-defined because the value of $t$ in $\Z_l$ determines its value modulo $l^{2k}$ via the map $\Z_l\rightarrow \Z/l^{2k}\Z$. Note that apart from this dependence, all other random variables live on different event spaces, and are thus independent (on the product event space).

By Lemma \ref{lem: basic v_l,k facts} and the fact that $Y_{l,p,k}$ and $Y_{l,p}$ stem from the same trace $t$, we have $\log(Y_{l,p,k}) = \log(Y_{l,p}) + O(l^{-k})$. Thus, in the same way as the proof of Proposition \ref{prop: second step}, we obtain for $k>1$ that
\begin{equation}\label{eqn: pf 2D eqn 4}
\log \left(\prod_{l<l_1} Y_{l,p,k}\right) = \log \left(\prod_{l<l_1} Y_{l,p}\right) + O(2^{-k}).
\end{equation}
By Lemma \ref{lem: Y_l,p expectation variance}, we have $\Var[\log(Y_{l,p})] = O(l^{-2})$, so by independence,
\begin{equation}\label{eqn: pf 2D eqn 5}
\Var\left[\log\left(\prod_{l\ge l_0}Y_{l,p}\right)\right] = \Var\left[\sum_{l\ge l_1}\log(Y_{l,k})\right] = \sum_{l\ge l_1}\Var\left[\log(Y_{l,k})\right] = O(l_1^{-1})
\end{equation}
as $\sum_{l\ge l_1}l^{-2} = O(l_1^{-1})$. In the same way as the earlier part of the proof, by combining Equations \eqref{eqn: pf 2D eqn 2}, \eqref{eqn: pf 2D eqn 4} and \eqref{eqn: pf 2D eqn 5}, and then using Chebyshev's inequality, we obtain that we almost always have
$$\frac{\sqrt{4-(t_0/\sqrt p)^2}}{2\pi}  \prod_{l<l_1}Y_{l,p,k} = \frac{\sqrt{4-X^2}}{2\pi} Z_p + o(1),$$
so likewise by Lemma \ref{lem: bound metric} we conclude that $\pi(\alpha_{p,t_0}',\alpha_{p,t_0}^{heu}) \rightarrow 0$ uniformly in $t_0$ as $p\rightarrow \infty$.
\end{proof}
To complete the proof of Theorem \ref{thm: 2D vertical sato-tate}, we decompose $[-2,2]\times [0,\infty)$ into small vertical slices $[t_0/\sqrt p,(t_0+T')/\sqrt p)\times [0,\infty)$ and then use Proposition \ref{prop: small vertical slices} which says that we have good convergence on these small vertical slices.
\begin{proof}[Proof of Theorem \ref{thm: 2D vertical sato-tate}]
We wish to show that for every $\epsilon>0$ that for sufficiently large $p$, $\pi(\rho_p,\rho_p^{heu})\le \epsilon$. Fix some $\epsilon>0$. By the definition of the Lévy-Prokhorov metric, we need to show that for all $A\in \mathcal B(\R^2)$ that both $\rho_p(A)\le \rho_p^{heu}(A^{\epsilon})+\epsilon$ and $\rho_p^{heu}(A)\le \rho_p(A^{\epsilon})+\epsilon$.

Both distributions are supported on $[-2,2]\times [0,\infty)$, which we split into small vertical slices as follows. Let $N=N(p) =\lfloor \frac{4\sqrt p}{T'}\rfloor$, and note that $N\rightarrow \infty$ as $p\rightarrow \infty$ because $T'=o(\sqrt p)$. Define the intervals $I_n = \left[\frac{nT'}{\sqrt p}-2, \frac{(n+1)T'}{\sqrt p}-2\right)$ and regions $C_n = I_n \times [0,\infty)$ for integers $0\le n \le N-1$. For convenience, denote the incomplete interval that was left over to be $I_N = \left[\frac{NT'}{\sqrt p}-2, 2\right]$ with the incomplete region $C_N = I_N\times [0,\infty)$. Note that $\{C_n\}_{0\le n\le N}$ are disjoint and cover all of $[-2,2]\times [0, \infty)$.

Fix any $A\in \mathcal B(\R^2)$. Let $q$ be the projection to the second component. Then, let $L_n = q(A\cap C_n)$ for $0\le n\le N$, so we have 
\begin{equation}\label{eqn: slices to 2D proof eqn 1}
A\cap C_n \subseteq I_n \times q(A\cap C_n).
\end{equation}
On the other hand, we claim that for sufficiently large $p$ that 
\begin{equation}\label{eqn: slices to 2D proof eqn 2}
A^\epsilon \cap C_n \supseteq I_n \times (q(A\cap C_n))^{\epsilon/2}
\end{equation} for all $0\le n\le N$. Indeed, take $(x,y) \in I_n \times (q(A\cap C_n))^{\epsilon/2}$, then there exists some $y'$ with $|y'-y|\le \epsilon /2$ where $y'\in q(A\cap C_n)$. By definition, this means that there is a point $(x',y') \in A\cap C_n$. Since both $x,x'\in I_n$, $|x-x'|\le T'/\sqrt p = o(1)$. Thus, the distance between $(x,y)$ and $(x',y')$ is at most $\sqrt{o(1)^2+(\epsilon/2)^2}<\epsilon$ for sufficiently large $p$. Hence, $(x,y)\in A^\epsilon$ and so $(x,y)\in A^\epsilon \cap C_n$. 

Thus, we obtain the following for sufficiently large $p$.
\begin{equation*}
\begin{split}
\rho_p(A) &= \sum_{n=0}^N \rho_p(A\cap C_n) \\
&\le \sum_{n=0}^{N} \rho_p(I_n \times q(A\cap C_n)) \\
&= \rho_p(I_N \times q(A\cap C_N)) + \sum_{n=0}^{N-1} \frac{T'}{4\sqrt p}\alpha_{p,nT'}(q(A\cap C_n))\\
&\le \frac{T'}{4\sqrt p} + \sum_{n=0}^{N-1} \frac{T'}{4\sqrt p} \left[\alpha_{p,nT'}^{heu}\left((q(A\cap C_n))^{\epsilon/2}\right)+\epsilon/2\right]\\
&< \epsilon + \sum_{n=0}^{N-1} \frac{T'}{4\sqrt p} \alpha_{p,nT'}^{heu}\left((q(A\cap C_n))^{\epsilon/2}\right)\\
&= \epsilon + \sum_{n=0}^{N-1}\rho_p^{heu}\left(I_n \times (q(A\cap C_n))^{\epsilon/2}\right)\\
&\le \epsilon + \sum_{n=0}^{N-1}\rho_p^{heu}\left(A^{\epsilon}\cap C_n\right)\\
&\le \epsilon + \rho_p^{heu}(A^\epsilon).
\end{split}
\end{equation*}
The first line follows as $A$ is a disjoint union of $A\cap C_n$, and the second follows from Equation \eqref{eqn: slices to 2D proof eqn 1}. 
For the third line, we have $\rho_p(I_n \times q(A\cap C_n)) = \frac{T'}{4\sqrt p}\alpha_{p,nT'}(q(A\cap C_n))$ because there is a $T'/4\sqrt p$ probability that the normalized trace $t/\sqrt p$ lies in $I_n$, and by definition, $\alpha_{p,nT'}(q(A\cap C_n))$ is the probability that $N_{t,p}/2\sqrt p$ lies in $q(A\cap C_n)$ given that $t/\sqrt p$ lies in $I_n$. For the fourth line, we use a trivial bound for $\rho_p(I_N\times q(A\cap C_n))$, and we invoke Proposition \ref{prop: small vertical slices} which says that for sufficiently large $p$, $\pi(\alpha_{p,nT'},\alpha_{p,nT'}^{heu})<\epsilon/2$ for all $n$. Unravelling the definition of the Lévy-Prokhorov metric gives us the inequality. In the fifth line, we combine all the error terms as $T'/4\sqrt p<\epsilon/2$ for sufficiently large $p$, and the other errors are $(\sum_{n=0}^N \frac{T'}{4\sqrt p})\frac{\epsilon}{2}\le  \epsilon/2$. The explanation for the sixth line is similar to that for the third line, and the remaining lines follow from Equation \eqref{eqn: slices to 2D proof eqn 2} and the fact that $A^{\epsilon}$ is a disjoint union of $A^{\epsilon}\cap C_n$.

Lastly, we see that we can swap $\rho_p$ with $\rho_p^{heu}$, as well as $\alpha_{p,nT'}$ with $\alpha_{p,nT'}^{heu}$ to get $\rho_p^{heu}(A)\le \rho_p(A^{\epsilon})+\epsilon$ for all sufficiently large $p$. This is because the statement of Proposition \ref{prop: small vertical slices} is symmetric in $\alpha_{p,t_0}$ and $\alpha_{p,t_0}^{heu}$. 
\end{proof}
\subsection{Convergent subsequences of measures}\label{sec: convergent subsequences}
Now that we have proven Theorem \ref{thm: 2D vertical sato-tate}, we know that the distributions $\rho_p$ and $\rho_p^{heu}$ are close as $p\rightarrow \infty$. Since there is an explicit formula for $\rho_p^{heu}$, this tells us what $\rho_p$ look like in the $p\rightarrow \infty$ limit. In this section we discuss the convergence of the distributions $\rho_p$, or more generally for subsequences $\rho_{p_i}$. We will show that the distributions $\rho_p$ do not converge, but also that some subsequences $\rho_{p_i}$ do. Furthermore, we give a sufficient condition on subsequences $p_i$ such that the distributions $\rho_{p_i}$ converge. 

The first step to proving these results is to use Theorem \ref{thm: 2D vertical sato-tate} to reduce the convergence of $\rho_{p_i}$ to the convergence of $Z_{p_i}$.
\begin{corollary}\label{cor: converge equivalence lemma}
Let $p_1,p_2,\ldots$ be an increasing sequence of primes. The following are equivalent:
\begin{enumerate}[(\alph*)]
\item The sequence of distributions $\rho_{p_i}$ converges.
\item The sequence of distributions $\rho_{p_i}^{heu}$ converges.
\item The sequence of distributions $Z_{p_i}$ converges.
\end{enumerate}
\end{corollary}
\begin{proof}
(a) and (b) are equivalent as by Theorem \ref{thm: 2D vertical sato-tate} we have $\pi(\rho_{p_i},\rho_{p_i}^{heu}) \rightarrow 0$ as $p_i\rightarrow \infty$, and because the Lévy-Prokhorov metric metricizes convergence in distribution for random variables. It is also clear that (b) and (c) are equivalent because $v_{p_i}^{heu}$ is defined to be the joint distribution of $(X,\frac{1}{2\pi}\sqrt{4-X^2}Z_p)$.
\end{proof}
Recall that $Z_{p} \coloneqq \prod_l Y_{l,p}$ is the infinite independent product of $Y_{l,p}$. For each $l$, there are only a few possible distributions for $Y_{l,p}$ as seen in Lemma \ref{lem: Y_l,p explicit distributions}. This motivates the following notation.
\begin{definition}
For odd $l$, define $Y_{l,0}'$, $Y_{l,-1}'$ and $Y_{l,+1}'$ to be random variables that have the distribution stated in part (a), (b) and (c) of the odd $l$ part of Lemma \ref{lem: Y_l,p explicit distributions} respectively. By this definition, the distribution of $Y_{l,p}$ is the same as the distribution of $Y'_{l,(\frac p l)}$.

For $l=2$, define $Y_{2,a}', Y_{2,b}', Y_{2,c}',Y_{2,d}'$ to be random variables corresponding to the distributions stated in part (a), (b), (c) and (d) of the $l=2$ part of Lemma \ref{lem: Y_l,p explicit distributions} respectively. These correspond to distributions of $Y_{2,p}$ for the $p=2$, $p\equiv 3\pmod 4$, $p \equiv 5\mod 8$ and $p\equiv 1\mod 8$ cases respectively.
\end{definition}
For convenience, introduce the function $f_2\colon \{\text{primes}\}\rightarrow \{a,b,c,d\}$ which sends the prime $p$ to the respective case as described above. We also extend this to other primes by setting $f_l\colon \{\text{primes}\} \rightarrow \{-1,0,+1\}$ sending $p$ to $(\frac p l)$. Then, by the above definition we have
$$Z_p = \prod_l Y'_{l,f_l(p)}$$
which is an infinite independent product.
\begin{proposition}\label{prop: Z_p_i converges}
Let $p_1,p_2,\ldots$ be an increasing sequence of primes. Suppose that for each prime $l$ that the value of $f_l(p_i)$ is eventually constant, and let this value be $e_l$. Then, 
$$Z_{p_i}\rightarrow \prod_l Y'_{l,e_l}$$
converges to the infinite independent product in distribution as $i\rightarrow \infty$.
\end{proposition}
\begin{proof}
It suffices to show that the logarithm of both sides converge, that is
$$\sum_{l} \log(Y'_{l,f_l(p_i)}) \rightarrow \sum_l \log(Y_{l,e_l}'),$$
where we use the infinite independent sum on both sides.
Note that by Lemma \ref{lem: Y_l,p expectation variance}, both $\Var[\log(Y'_{l,f_l(p_i)})]=O(l^{-2})$ and $\Var[\log(Y'_{l,e_l})]=O(l^{-2})$. Since the sum $\sum_l l^{-2}$ converges, there exists some $l'$ where 
$$\Var\left[\sum_{l\ge l'} \log(Y'_{l,f_l(p_i)}) - \sum_{l\ge l'} \log(Y_{l,e_l}')\right] = \sum_{l\ge l'} \Var[\log(Y'_{l,f_l(p_i)})] + \sum_{l\ge l'} \Var[\log(Y_{l,e_l}')] < \epsilon^3/8.$$
Again by Lemma \ref{lem: Y_l,p expectation variance}, we can repeat the same for expectation (adjusting $l'$ if necessary) to obtain 
$$\left| \E\left[\sum_{l\ge l'} \log(Y'_{l,f_l(p_i)}) - \sum_{l\ge l'} \log(Y_{l,e_l}')\right]\right|
< \epsilon /2.$$
By our assumption, there exists an $N$ such that for each $l<l'$, the term $f_l(p_i)=e_l$ for all $i\ge N$. Thus, for $i\ge N$,
$$\sum_l \log(Y'_{l,f_l(p_i)}) - \sum_l \log(Y_{l,e_l}') = \sum_{l\ge l'} \log(Y'_{l,f_l(p_i)}) - \sum_{l\ge l'} \log(Y_{l,e_l}').$$
Then, applying Corollary \ref{cor: variance metric} to the above three equations, we obtain that 
$$\pi\left(\mu_{(\sum_{l} \log(Y_{l,p_i}))}, \mu_{(\sum_l \log(Y_{l,e_l}'))} \right)< \epsilon$$
for all $i\ge N$.
\end{proof}
This directly implies Corollary \ref{cor: prime sequence converge}.
\begin{proof}[Proof of Corollary \ref{cor: prime sequence converge}]
By Corollary \ref{cor: converge equivalence lemma}, it suffices to show that $Z_{p_i}$ converges when $p_i$ satisfies the conditions. But the conditions imply that $f_l(p_i)$ is eventually constant, so Proposition \ref{prop: Z_p_i converges} implies that $Z_{p_i}$ converges in distribution.
\end{proof}
Note that in the statement of Corollary \ref{cor: prime sequence converge} we left the $p=2$ case out because an increasing sequence of primes cannot be eventually $2$. We also note that for each odd prime $l$, the Legendre symbols $(\frac{p_i}{l})$ cannot eventually be $0$. Next, we prove that the sequence $\rho_p$ for all primes does not converge.
\begin{proposition}
The sequence of distributions $\rho_p$ do not converge.
\end{proposition}
\begin{proof}
It suffices to exhibit two different sequences of primes that converge to different distributions. Let $p_1,p_2,\ldots$ be an increasing sequence of primes such that $f_2(p_i)=b$ and for each $l$ we eventually have $f_l(p_i) = 1$ -- this is possible by Dirichlet's theorem. Likewise, let $q_1,q_2,\ldots$ be an increasing sequence of primes such that $f_2(p_i)=c$ and for each $l$ we eventually have $f_l(p_i) = 1$. By Proposition \ref{prop: Z_p_i converges}, $Z_{p_i}$ converges in distribution to $Y_{2,b}'\prod_{l>2}Y_{l,1}'$ while $Z_{q_i}$ converges in distribution to $Y_{2,c}'\prod_{l>2}Y_{l,1}'$. It is clear that these two distributions are different.
\end{proof}
Now, we discuss the converse of Corollary \ref{cor: prime sequence converge}. We conjecture the following.
\begin{conjecture}\label{conj: prod Y not same}
For any two distinct sequences $\{e_l\}_{l \text{ prime}}$ and $\{e'_l\}_{l \text{ prime}}$, the distributions of the infinite independent products $\prod_l Y'_{l,e_l}$ and $\prod_l Y'_{l,e_l'}$ are distinct.
\end{conjecture}
Surprisingly, even though the distributions of $Y'_{l,e_l}$ are given explicitly in Lemma \ref{lem: Y_l,p explicit distributions}, this still seems difficult to prove as we are dealing with infinite products. This conjecture is very likely to be true as it would be a big coincidence for two infinite products of different distributions to be the same. We show that Conjecture \ref{conj: prod Y not same} implies the converse of Corollary \ref{cor: prime sequence converge}.
\begin{proposition}
Suppose that Conjecture \ref{conj: prod Y not same} is true. Then, the converse to Corollary \ref{cor: prime sequence converge} is true, i.e. $\rho_{p_i}$ converges in distribution only if for each odd prime $l$ the Legendre symbols $(\frac {p_i} l)$ is eventually constant, and additionally $p_i$ is eventually $3 \pmod 4$, $1 \pmod 8$, or $5 \pmod 8$.
\end{proposition}
\begin{proof}
Suppose to the contrary that there exists some prime $m$ where $f_{m}(p_i)$ is not eventually constant. Then there exist two different values of $f_{m}(p_i)$ that appear infinitely often, say $e_{m}\neq e_{m}'$, so there exist disjoint infinite subsequences $q_i$ and $r_i$ of $p_i$ where $f_{m}(q_i)=e_{m}$ and $f_{m}(r_i)=e_{m}'$ for all $i$. 

By the pigeonhole principle, there exists a further subsequence $q_i'$ of $q_i$, where $f_{2}(q_i') = e_2$ for some $e_2$, a further subsequence of $q_i''$ of $q_i'$ where $f_{3}(q_i'')=e_3$ for some $e_3$, and so on. Iterating this argument, we can obtain an infinite sequence $e_l$ indexed by primes $l$ satisfying the following condition: for any $s$, there exists an infinite subsequence $q_i^{s}$ of $q_i$ where $f_l(q_i)=e_l$ for all $l<s$. We repeat the same for $r_i$ to obtain a sequence $e_l'$ with the same property and define $r_i^s$ in the same way.

Since $e_{m}\neq e_{m}'$, the distributions of $\sum_l \log(Y_{l,e_l}')$ and $\sum_l \log(Y_{l,e_l'}')$ differ by Conjecture \ref{conj: prod Y not same}. Let the Lévy-Prokhorov distance between these two distributions be $3 \epsilon>0$. Repeating the argument in the proof of Proposition \ref{prop: Z_p_i converges}, pick $s$ to be sufficiently large such that the following bounds hold for all sufficiently large $i$:
$$\pi\left(\mu_{(\sum_{l} \log(Y_{l,q_i^{s}}))}, \mu_{(\sum_l \log(Y_{l,e_l}'))} \right)< \epsilon,$$
$$\pi\left(\mu_{(\sum_{l} \log(Y_{l,r_i^{s}}))}, \mu_{(\sum_l \log(Y_{l,e_l'}'))} \right)< \epsilon.$$
By the triangle inequality, 
$$\pi(\mu_{\log Z_{q_i^{s}}}, \mu_{\log Z_{r_i^{s}}}) = \pi\left(\mu_{(\sum_{l} \log(Y_{l,q_i^{s}}))}, \mu_{(\sum_{l} \log(Y_{l,r_i^{s}}))} \right) > 3\epsilon -\epsilon -\epsilon = \epsilon,$$
so $\log Z_{p_i}$ is not Cauchy in the metric $\pi$ and hence does not converge in distribution. Thus, $Z_{p_i}$ does not converge in distribution, and by Corollary \ref{cor: converge equivalence lemma}, $\rho_{p_i}$ does not converge in distribution.
\end{proof}

\subsection{Distribution of size of isogeny classes}\label{sec: corollaries}
Recall that the isogeny classes of elliptic curves over finite fields correspond to their Frobenius trace. Thus, the sizes of isogeny classes over $\F_p$ satisfy $N_{t,p}' \approx N_{t,p}$ for $t\in [-2\sqrt p, 2\sqrt p]$. This appears as the second component in $\rho_p$, so projecting the results of Theorem \ref{thm: 2D vertical sato-tate} to the second component directly gives us corresponding results on the distribution of size of isogeny classes. 
\begin{corollary}\label{cor: isogeny classes}
Let $s_p$ be the distribution of the size of isogeny classes of elliptic curves over $\F_p$. Also, let $X$ be the uniform random variable on $[-2,2]$ and suppose that $X$ and $Z_p$ are independent random variables. Then, $$\pi\left(\frac{s_p}{\sqrt p},\frac{\sqrt{4-X^2}}{\pi}Z_p\right)\rightarrow 0$$ 
as $p\rightarrow \infty$, where $\pi$ is the Lévy-Prokhorov metric. 

Furthermore, $s_p/\sqrt p$ do not converge in distribution as $p\rightarrow \infty$. For an increasing sequence of primes $p_i$, $s_{p_i}/\sqrt{p_i}$ converges in distribution if for each odd prime $l$ the Legendre symbols $(\frac{p_i}{l})$ is eventually constant, and $p_i$ is either eventually $3\pmod 4$, eventually $1\pmod 8$ or eventually $5\pmod 8$.
\end{corollary}
\section{Stronger vertical Sato-Tate}\label{sec: pf of stronger sato-tate}
The goal of this section is to prove Theorem \ref{thm: stronger vertical sato-tate}. Throughout this section, we fix a function $f(p)$, then we choose a prime $p$ and an interval $I=[t_0,t_0+f(p))\subseteq [-2\sqrt p, 2\sqrt p]$. Pick a random trace $t\in I$, and we carry out the method described in Section \ref{sec: proof overview strong} in reverse order. For the statement assuming GRH we will use $l_0 = \log(p)^{2+\epsilon}$, and for the unconditional statement we will use $l_0 = g_\kappa(p)$ for a suitable choice of $\kappa$ as in Proposition \ref{prop: second step}.

\begin{lemma}\label{lem: expanded product}
Suppose we are given distinct primes $l_1,l_2,\ldots ,l_m<l_0$, and positive integers $\alpha_i$ where $\sum_{i=1}^m \alpha_i = n$. Then there exists some universal constant $C$ where
\begin{equation}\label{eqn: expanded product}
\E\left[\prod_{i=1}^m \log(v_{l_i,k}(\Delta))^{\alpha_i}\right] = \E\left[\prod_{i=1}^m \log(Y_{l_i,p,k})^{\alpha_i}\right] + O\left(\frac{C^n l_0^{2kn}}{f(p)}\right)
\end{equation}
where we take the infinite independent product on the right.
\end{lemma}
\begin{proof}
The proof is similar to the proof of Lemma \ref{lem: individual var and covar}. Since each $\log(v_{l_i,k}(\Delta))^{\alpha_i}$ is periodic in $t$ with period dividing $l_i^{2k}$, the terms $\log(v_{l_i,k}(\Delta))^{\alpha_i}$ are independent in an interval of length $T=\prod_{i=1}^m l_i^{2k}<l_0^{2kn}$. Then, we split the interval $I$ into intervals of length $T$ with a leftover interval of length $<T$. On the intervals of exactly length $T$ the two expectations agree, so the error comes from the leftover interval. The probability that $t$ is in the leftover interval is less than $T/f(p)$. By definition of $v_{l,k}$ the factor $\log(v_{l,k}(\Delta))$ is bounded, so let $C$ be a universal constant such that $\log(v_{l,k}(\Delta))\le C$. Thus, we have both $\prod_{i=1}^m \log(v_{l_i,k}(\Delta))^{\alpha_i} \le C^n$ and $\prod_{i=1}^m \log(Y_{l_i,p,k})^{\alpha_i} \le C^n$. Taking the product of $T/f(p)$ and $2C^n$, we obtain the error term.
\end{proof}
\begin{lemma}\label{lem: power n bound}
There exists some universal constant $C$ where
\begin{equation}\label{eqn: power n bound}
\E\left[\left(\sum_{l<l_0} \log(v_{l,k}(\Delta))\right)^n\right] = \E\left[\left(\sum_{l<l_0} \log(Y_{l,p,k})\right)^n\right] + O\left(\frac{C^nl_0^{(2k+1)n}}{f(p)}\right)
\end{equation}
where we take the independent sum on the right. 
\end{lemma}
\begin{proof}
We obtain this equation by summing up Equation \eqref{eqn: expanded product} over all possible integers $m$, distinct primes $l_1,\cdots l_m < l_0$ and integers $\alpha_1,\ldots, \alpha_m$ such that $\sum_{i=1}^m \alpha_i = n$. There are $\#\{\text{primes } l < l_0\}^n < l_0^n$ such combinations, so the error term is $l_0^n$ multiplied by the error in Equation \eqref{eqn: expanded product}.
\end{proof}
However, this bound is not very good when $n$ is large. This is because the periodicity argument in the proof of Lemma \ref{lem: expanded product} is ineffective when the period $T$ is greater than $f(p)$. Hence, for large $n$ we will use the following bound which is independent of $f(p)$ instead.
\begin{lemma}\label{lem: different power n bound}
There exists some universal constant $C$ where
\begin{equation}\label{eqn: different power n bound}
\E\left[\left(\sum_{l<l_0} \log(v_{l,k}(\Delta))\right)^n\right] = \E\left[\left(\sum_{l<l_0} \log(Y_{l,p,k})\right)^n\right] + O(C^n \log(l_0)^n)
\end{equation}
where we take the independent sum on the right. 
\end{lemma}
\begin{proof}
Since $\log(v_{l,k})(\Delta) = O(l^{-1})$ and $\sum_{l<l_0} l^{-1} = O(\log(l_0))$, we see that 
$\sum_{l<l_0} \log(v_{l,k}(\Delta)) = O(\log(l_0))$. Thus, there exists some constant $C$ where $\left(\sum_{l<l_0} \log(v_{l,k}(\Delta))\right)^n = O(C^n\log(l_0)^n)$, and likewise we have $\left(\sum_{l<l_0} \log(Y_{l,p,k})\right)^n = O(C^n\log(l_0)^n)$, giving us the desired bound.
\end{proof}
Combining both bounds gives us the following proposition.
\begin{proposition}\label{prop: prod v_l,k bound}
There exists some universal constant $C'$ where if $f(p)\ge \exp(C'k(\log(l_0)^2))$ for all sufficiently large $p$, then
\begin{equation}\label{eqn: prod v_l,k bound}
\E\left[\prod_{l<l_0}v_{l,k}(\Delta)\right] = \E\left[\prod_{l<l_0}Y_{l,p,k}\right] + o(1)
\end{equation}
where we take the independent product on the right.
\end{proposition}
\begin{proof}
By Taylor expansion (see Equation \eqref{eqn: taylor expansion}), this equation is obtained by summing up Equation \eqref{eqn: power n bound} or Equation \eqref{eqn: different power n bound} for all integers $n\ge 0$, weighted by $\frac{1}{n!}$. Choose a sufficiently large constant $C$ such that both Equation \eqref{eqn: power n bound} and \eqref{eqn: different power n bound} hold with this $C$. We consider the following two ranges for $n$.
\begin{enumerate}[(\alph*)]
    \item For $n\ge 3C\log(l_0)$, we will use Equation \eqref{eqn: different power n bound}, so the error term coming from this range of $n$ is
    $O\left(\sum_{n=3C\log(l_0)}^\infty \frac{C^n\log(l_0)^n}{n!} \right).$
    However, note that in the sum, each successive term is less than half of the previous term, so the first term dominates and the total error is $O(\frac{(C\log(l_0))^{3C\log(l_0)}}{(3C\log(l_0))!})=o(1)$ as $p\rightarrow \infty$ by Stirling's approximation.
    \item For $n<3C\log(l_0)$, we will use Equation \eqref{eqn: power n bound} instead, so the error term coming from this range of $n$ is
    $$O\left(\frac{\sum_{n=0}^{3C\log(l_0)-1} C^n l_0^{(2k+1)n}}{f(p)}\right) = O\left(3C\log(l_0)\frac{(Cl_0^{2k+1})^{3C\log(l_0)}}{f(p)}\right) = O\left(\frac{e^{C'k(\log(l_0))^2}}{f(p)}\right) = o(1)$$
    for some constant $C'$, and by our assumption on $f(p)$.
\end{enumerate}
As the errors for both ranges of $n$ are $o(1)$, we have the desired result.
\end{proof}
Lastly, combining Proposition \ref{prop: second step} and \ref{prop: prod v_l,k bound} yields Theorem \ref{thm: stronger vertical sato-tate}.
\begin{proof}[Proof of Theorem \ref{thm: stronger vertical sato-tate}]
We first tackle the case assuming GRH. Take $l_0=(\log p)^3$. The condition implies that for all constants $C$ we have $f(p)\ge \exp(C(\log(l_0))^2)$ for all sufficiently large $p$, so we can choose $k=k(p)$ to be a slowly increasing unbounded function such that $f(p)\ge \exp(C'k(\log(l_0)^2))$ for all sufficiently large $p$, where $C'$ is the constant in Proposition \ref{prop: prod v_l,k bound}. Thus, by Proposition \ref{prop: second step} and \ref{prop: prod v_l,k bound} respectively, Equations \eqref{eqn: rounded off estimate} and \eqref{eqn: prod v_l,k bound} are true. Since $t\in [t_0,t_0+f(p))$ and $f(p)=o(\sqrt p)$, we have $$\frac{\sqrt{-\Delta}}{2\pi \sqrt p} \rightarrow \frac{1}{2\pi}\sqrt{4-(t_0/\sqrt p)^2}$$
as $p\rightarrow \infty$, so substituting this into Equation \eqref{eqn: rounded off estimate} and taking expectation yields
$$\E\left[\frac{N_{t,p}}{2\sqrt p}\right] = (1+o(1))\frac{1}{2\pi}\sqrt{4-(t_0/\sqrt p)^2} \ \E\left[\prod_{l<l_0}v_{l,k}(\Delta)\right].$$
Then, using Equation \eqref{eqn: prod v_l,k bound} gives
$$\E\left[\frac{N_{t,p}}{2\sqrt p}\right] = (1+o(1))\frac{1}{2\pi}\sqrt{4-(t_0/\sqrt p)^2} \left(\E\left[\prod_{l<l_0}Y_{l,p,k}\right]+o(1)\right)$$
where we take the independent product on the right. 

Define $Y_{l,p,k}$ to be dependent on $Y_{l,p}$ in the same way as in the proof of Proposition \ref{prop: small vertical slices}, so by Equation \eqref{eqn: pf 2D eqn 4} we have 
$$\E\left[\prod_{l<l_0}Y_{l,p,k}\right] = \E\left[\prod_{l<l_0}Y_{l,p}\right] + o(1).$$By Lemma \ref{lem: Y_l,p expectation variance}, $\E[Y_{l,p}] = 1$ for $p\neq l$ and $\E[Y_{p,p}] = 1+O(p^{-1})$, so taking the product gives $$\E\left[\prod_{l<l_0}Y_{l,p}\right] = 1+o(1).$$
Hence, combining the equations above yields $$\E\left[\frac{N_{t,p}}{2\sqrt p}\right] = (1+o(1)) \frac{1}{2\pi}\sqrt{4-(t_0/\sqrt p)^2},$$
and this directly implies Equation \eqref{eqn: threshold question} because
$$\frac{\#\{E/\F_p \mid t_p(E)\in \left[t,t+f(p)\right)\}}{\frac{f(p)}{\sqrt p}\#\{E/\F_p\}} = \frac{f(p)\E[N_{t,p}] }{\frac{f(p)}{\sqrt p}2p} + o(1)=  \E\left[\frac{N_{t,p}}{2\sqrt p}\right] + o(1),$$
and the error term is uniform across all $t$.

Now, we prove the unconditional statement without assuming GRH in almost the same way. The condition tells us that there is some $\epsilon'>0$ where $f(p)\ge p^{\epsilon'}$ for all sufficiently large $p$. In Proposition \ref{prop: second step}, pick $\epsilon=\epsilon'/3$ and take $l_0 = g_{\kappa}(p)$, so there are at most $O(p^{\epsilon})$ exceptional values of $t$ that do not satisfy Equation \eqref{eqn: rounded off estimate}. Hence, there is a $O(p^{-2\epsilon})$ probability for a random $t\in [t_0,t_0+f(p))$ to be an exception. For these exceptions, we use the following well-known bound that for every $\epsilon>0$ we have $H(\Delta) = O(|\Delta|^{1/2+\epsilon})$, see Corollary 5.2 of \cite{Katz_Lang-Trotter}. Thus, $N_{t,p} = H(\Delta) = O(p^{1/2+\epsilon})$ so the exceptions will only affect $\E\left[\frac{N_{t,p}}{2\sqrt p}\right]$ by at most $O(p^{-\epsilon})=o(1)$. Then, our choice of $f(p)$ and $l_0$ satisfy the conditions of both the unconditional version of Proposition \ref{prop: second step} and \ref{prop: prod v_l,k bound}, so Equation \eqref{eqn: rounded off estimate} and \eqref{eqn: prod v_l,k bound} are true, and we proceed in the same way as above.
\end{proof}
\section{Appendix}
The following lemma states the possible distributions for $Y_{l,p}$:

\begin{lemma}\label{lem: Y_l,p explicit distributions}
For an odd prime $l$, there are the following three cases.
\begin{enumerate}[(\alph*)]
\item If $p=l$, then 
$$Y_{l,p} = \begin{cases}
(1-\frac 1 l)^{-1} & \text{with probability } (l-1)/l \\
1 & \text{with probability } 1/l\\
\end{cases}.$$
\item If $(\frac p l)=-1$, then 
$$
Y_{l,p} = \begin{cases}
(1+\frac 1 l)^{-1} & \text{with probability }(l+1)/2l \\
(1-\frac 1 l)^{-1} & \text{with probability }(l-1)/2l \\
\end{cases}.$$
\item If $(\frac p l)=+1$, then $$
Y_{l,p} = \begin{cases}
(1+\frac 1 l)^{-1} & \text{with probability }(l-1)/2l \\
(1-\frac 1 l)^{-1} & \text{with probability }(l^2-2l-1)/2(l^2+l) \\
(l-\frac{1}{l^{s-1}})/(l-1) & \text{with probability }2(l-1)/l^{2s}\text{ for all } s\ge 1 \\
(l-\frac{2}{l^s+l^{s-1}}))/(l-1) & \text{with probability }(l-1)/l^{2s+1}\text{ for all } s\ge 1 \\
\end{cases}.$$
\end{enumerate}
%https://www.wolframalpha.com/input?i=%282+l%5E3%29%2F%28%28-1+%2B+l%29+%281+%2B+l%29+%281+%2B+l+%2B+l%5E2%29%29+%2B+%281%2B1%2Fl%29%5E%7B-1%7D*%28l-1%29%2F%282l%29+%2B+%281-1%2Fl%29%5E%7B-1%7D%28l%5E2-2l-1%29%2F%282l%5E2%2B2l%29+%2B+%28-1+%2B+l+%2B+l%5E2%29%2F%28%28-1+%2B+l%29+%281+%2B+l%29+%281+%2B+l+%2B+l%5E2%29%29
%https://www.wolframalpha.com/input?i=%28l-1%29%2F%282l%29+%2B+%28l%5E2-2l-1%29%2F%282l%5E2%2B2l%29+%2B+%282%29%2F%28l%2B1%29+%2B+1%2F%28l%5E2%2Bl%29
If $l=2$, there are the following cases.
\begin{enumerate}[(\alph*)]
\item If $p=2$, then 
$$Y_{2,p} = \begin{cases}
1 & \text{with probability } 1/2\\
2 & \text{with probability } 1/2\\
\end{cases}.$$
\item If $p\equiv 3 \pmod 4$, then 
$$Y_{2,p} = \begin{cases}
2/3 & \text{with probability } 1/2\\
1 & \text{with probability } 1/4\\
4/3 & \text{with probability } 1/8\\
2 & \text{with probability } 1/8\\
\end{cases}.$$
\item If $p\equiv 5 \pmod 8$, then
$$Y_{2,p} = \begin{cases}
2/3 & \text{with probability } 1/2\\
1 & \text{with probability } 1/4\\
3/2 & \text{with probability } 1/8\\
5/3 & \text{with probability } 1/16\\
2 & \text{with probability } 1/16\\
\end{cases}.$$
\item If $p\equiv 1 \pmod 8$, then
$$Y_{2,p} = \begin{cases}
2/3 & \text{with probability } 1/2\\
1 & \text{with probability } 1/4\\
3/2 & \text{with probability } 1/8\\
2 & \text{with probability } 1/48\\
2-\frac{1}{2^{s+1}} & \text{with probability } 1/2^{2s+2}\text{ for all } s\ge 1\\
2-\frac{1}{3\cdot 2^{s}} & \text{with probability } 1/2^{2s+4} \text{ for all } s \ge 1\\
\end{cases}.$$
\end{enumerate}
\end{lemma}
%https://www.wolframalpha.com/input?i=2%2F3*1%2F2+%2B+1%2F4+%2B+3%2F2*1%2F8+%2B+2%2F48+%2B+25%2F168%2B13%2F336
\begin{proof}
This follows from the definition of $Y_{l,p}$. In particular, the $p\neq l$ cases have been worked out in pages 9-10 of \cite{galbraith_mckee_2000}.
\end{proof}

\printbibliography

\end{document}